\documentclass[12pt,american]{amsart}
\usepackage[T1]{fontenc}
\usepackage[latin9]{inputenc}
\usepackage[letterpaper]{geometry}
\usepackage{color}
\usepackage{enumitem}
\usepackage{amstext}
\usepackage{amsthm}
\usepackage{amssymb}
\usepackage{graphicx} 

\makeatletter

\providecommand{\tabularnewline}{\\}

\numberwithin{equation}{section}
\numberwithin{figure}{section}
\numberwithin{table}{section}
\theoremstyle{plain}
\newtheorem{thm}{\protect\theoremname}[section]
  \theoremstyle{plain}
  \newtheorem{prop}[thm]{\protect\propositionname}
  \theoremstyle{definition}
  \newtheorem{example}[thm]{\protect\examplename}
  \theoremstyle{plain}
  \newtheorem{cor}[thm]{\protect\corollaryname}
  \theoremstyle{definition}
  \newtheorem{defn}[thm]{\protect\definitionname}
  \theoremstyle{plain}
  \newtheorem{lem}[thm]{\protect\lemmaname}
  \theoremstyle{plain}
  \newtheorem{conjecture}[thm]{\protect\conjecturename}

\usepackage{tikz}
\usepackage{graphicx}
\usetikzlibrary{trees,positioning,arrows,snakes}
\makeatletter
\def\env@cases{%
  \let\@ifnextchar\new@ifnextchar
  \left\lbrace
  \def\arraystretch{1}%
  \array{@{}l@{\quad}l@{}}}
\makeatother

\setlist{itemsep=0pt,topsep=0pt,parsep=1pt,partopsep=0pt}

\makeatother

\usepackage{babel}
  \providecommand{\conjecturename}{Conjecture}
  \providecommand{\corollaryname}{Corollary}
  \providecommand{\definitionname}{Definition}
  \providecommand{\examplename}{Example}
  \providecommand{\lemmaname}{Lemma}
  \providecommand{\propositionname}{Proposition}
\providecommand{\theoremname}{Theorem}

\begin{document}

\curraddr{Northern Arizona University, Department of Mathematics and Statistics,
Flagstaff, AZ 86011-5717, USA}

\email{dana.ernst@nau.edu}

\email{nandor.sieben@nau.edu}



\title[Impartial achievement and avoidance games for generating finite groups]{Impartial achievement and avoidance games\\
 for generating finite groups}

\author{Dana C. Ernst}

\author{N\'andor Sieben}

\date{\the\month/\the\day/\the\year}
\begin{abstract}
We study two impartial games introduced by Anderson and Harary and
further developed by Barnes. Both games are played by two players
who alternately select previously unselected elements of a finite
group. The first player who builds a generating set from the jointly
selected elements wins the first game. The first player who cannot
select an element without building a generating set loses the second
game. After the development of some general results, we determine
the nim-numbers of these games for abelian and dihedral groups. We
also present some conjectures based on computer calculations. Our
main computational and theoretical tool is the structure diagram of
a game, which is a type of identification digraph of the game digraph
that is compatible with the nim-numbers of the positions. Structure
diagrams also provide simple yet intuitive visualizations of these
games that capture the complexity of the positions. 
\end{abstract}

\maketitle
\global\long\def\gen{\text{\sf GEN}}

\global\long\def\dng{\text{\sf DNG}}

\global\long\def\mex{\operatorname{mex}}

\global\long\def\nim{\operatorname{nim}}

\global\long\def\opt{\operatorname{Opt}}

\global\long\def\nopt{\operatorname{nOpt}}

\global\long\def\pty{\operatorname{pty}}

\global\long\def\type{\operatorname{type}}

\global\long\def\otype{\operatorname{otype}}

\global\long\def\Otype{\operatorname{Otype}}

\global\long\def\spr{\operatorname{spr}}

\global\long\def\ag{\operatorname{Grp}}

\global\long\def\snf{\operatorname{Snf}}

\global\long\def\tsnf{\operatorname{tSnf}}

\global\long\def\diag{\operatorname{diag}}

\section{Introduction}

Anderson and Harary \cite{anderson.harary:achievement} introduced
two impartial games in which two players alternately select previously
unselected elements of a finite group until the group is generated
by the chosen elements. The first player who builds a generating set
from the jointly selected elements wins the achievement game denoted
by $\gen$. The first player who cannot select an element without
building a generating set loses the avoidance game denoted by $\dng$. 

In the original description of the avoidance game given in~\cite{anderson.harary:achievement},
the game ends when a generating set is built. This suggests mis\`ere-play
convention. We want to study both the achievement and the avoidance
game under normal-play convention. So, our version of the avoidance
game does not allow the creation of a generating set, and the game
ends when there are no available moves. Our version of the avoidance
game has the same outcome as the original, and so the difference is
immaterial since Anderson and Harary do not consider game sums. 

The outcome of both games was determined for finite abelian groups
in~\cite{anderson.harary:achievement}. Barnes~\cite{Barnes} provides
criteria for determining the outcome of each game for an arbitrary
finite group. Barnes applies his criteria to determine the outcome
of some of the more familiar finite groups, including abelian, dihedral,
symmetric, and alternating groups, although his analysis is incomplete
for alternating groups in the avoidance game.

Brandenburg studies related games in~\cite{brandenburg:algebraicGames}.
In one of the variations, two players alternate moves, where a move
consists of picking some non-identity element from a finitely generated
abelian group. The game then continues with a group that results by
taking the quotient by the subgroup generated by the chosen element.
The player with the last possible move wins.

The fundamental problem in the theory of impartial combinatorial games
is the determination of the nim-number of the game. This allows for
the calculation of the nim-numbers of game sums and the determination
of the outcome of the games. The major aim of this paper is the development
of some theoretical tools that allow the calculation of the nim-numbers
of the achievement and avoidance games for a variety of familiar groups.

The paper is organized as follows. In Section~\ref{sec:Preliminaries},
we review the basic terminology of impartial games and establish our
notation. We further our general study of avoidance and achievement
games in Sections~\ref{sec:DNG} and \ref{sec:GEN}, respectively.
In particular, we introduce the structure diagram of a game, which
is an identification digraph of the game digraph that is compatible
with the nim-numbers of the positions. Structure diagrams also provide
simple but intuitive visualizations of these games that capture the
complexity of the positions. By making further identifications, we
obtain the simplified structure diagram of a game, which will be our
main computational and theoretical tool in the remainder of the paper.
The main result of Section~\ref{sec:DNG} states that the nim-number
of the avoidance game is 0, 1, or 3 for an arbitrary finite group
(see Corollary~\ref{cor:nim-value-dng-013}). Analogously, in Section~\ref{sec:GEN},
we show that if the order of a group is odd, then the nim-number of
the corresponding achievement game is 1 or 2 (see Corollary~\ref{cor:nimOddGEN12}).
We conjecture that if the group is of even order, then the nim-number
of the achievement game is in $\{0,1,2,3,4\}$ (see Conjecture~\ref{conj:nimEvenGEN01234}).
Section~\ref{sec:Algorithms} describes the algorithms we implemented
via a computer program to generate our initial conjectures and verify
our results. Sections~\ref{sec:Cyclic}, \ref{sec:Dihedral}, and
\ref{sec:Abelian} contain a complete analysis of the nim-numbers
for cyclic, dihedral, and abelian groups, respectively. In Section~\ref{sec:SymmetricAlternating},
we study the symmetric and alternating groups. In particular, we provide
a description of the nim-numbers for the avoidance game for the symmetric
groups. In addition, we provide a partial characterization for the
achievement game for the symmetric groups, as well as both games for
the alternating groups. We conclude with several open questions in
Section~\ref{sec:Questions}.

The authors thank Bret Benesh and the anonymous referee for suggestions
that greatly improved the paper.

\section{Preliminaries\label{sec:Preliminaries}}

We briefly recall the basic terminology of impartial games to introduce
our notation. A comprehensive treatment of impartial games can be
found in \cite{albert2007lessons,SiegelBook}. For our purposes, an
\emph{impartial game} is a finite set $X$ of \emph{positions} together
with a starting position and a collection $\{\opt(P)\subseteq X\mid P\in X\}$
of possible \emph{options}. Two players take turns choosing one of
the available options in $\opt(P)$ of the current position $P$.
The player who encounters an empty option set cannot move and therefore
\emph{loses}. All games must come to an end in finitely many turns,
so we do not allow infinite lines of play. There are two possible
\emph{outcomes} for an impartial game. The game is an \emph{N-position}
if the next player (i.e., the player that is about to move) wins and
it is a \emph{P-position} if the previous player (i.e., the player
that just moved) wins.

The \emph{minimum excludant} $\mex(A)$ of a set $A$ of ordinals
is the smallest ordinal not contained in the set. The\emph{ nim-number}
$\nim(P)$ of a position $P$ is the minimum excludant of the set
of nim-numbers of the options of $P$. That is, 
\[
\nim(P):=\mex(\nopt(P)),
\]
where $\nopt(P):=\{\nim(Q)\mid Q\in\opt(P)\}$. Note that the minimum
excludant of the empty set is $0$, and so the terminal positions
of a game have nim-number 0. The \emph{nim-number of a game} is the
nim-number of its starting position. The nim-number of a game determines
the outcome of a game since a position $P$ is a P-position if and
only if $\nim(P)=0$.

The \emph{sum} of the games $P$ and $R$ is the game $P+R$ whose
set of options is 
\[
\opt(P+R):=\{Q+R\mid Q\in\opt(P)\}\cup\{P+S\mid S\in\opt(R)\}.
\]
This means that in each turn a player makes a valid move either in
game $P$ or in game $Q$. The nim-number of the sum of two games
can be determined as the \emph{nim-sum} 
\[
\nim(P+R)=\nim(P)\oplus\nim(R),
\]
 which requires binary addition without carry.

We write $P=R$ if the outcome of $P+T$ and $R+T$ is the same for
every game $T$. The one pile NIM game with $n$ stones is denoted
by the \emph{nimber} $*n$. The set of options of $*n$ is $\opt(*n)=\{*0,\ldots,*(n-1)\}$.
The following fundamental result shows the significance of the nimbers.
\begin{thm}
\emph{(Sprague--Grundy)} If $P$ is an impartial game, then $P=*\nim(P)$.
\end{thm}
We now recall a few well-known group-theoretic results and definitions
that will be useful in the remainder of the paper. The subgroup of
a group $G$ generated by the subset $S$ is the intersection of all
subgroups of $G$ containing $S$. Note that the empty set generates
the trivial subgroup.

A maximal proper subgroup of a group $G$ is called a \emph{maximal
subgroup} of $G$. It is clear that every proper subgroup of a finite
group is contained in a maximal subgroup. Note that the finite requirement
is necessary. For example, the group $(\mathbb{Q},+)$ has no maximal
subgroups. Maximal subgroups play an important role for us because
of the following easy fact.
\begin{prop}
\label{prop:GeneratingSubset}A subset $S$ of a finite group is a
generating set if and only if $S$ is not contained in any maximal
subgroup. 
\end{prop}
Next, we provide a more precise overview of the achievement and avoidance
games. Let $G$ be a finite group. We define the avoidance game $\dng(G)$
as follows. The first player chooses $x_{1}\in G$ such that $\langle x_{1}\rangle\neq G$
and at the $k$th turn, the concerned player selects $x_{k}\in G\setminus\{x_{1},\ldots,x_{k-1}\}$,
such that $\langle x_{1},\ldots,x_{k}\rangle\neq G$. That is, a position
in $\dng(G)$ is a set of jointly selected elements that must be a
non-generating subset of $G$. The player who cannot select an element
without building a generating set loses the game. 

In the achievement game $\gen(G)$, the first player chooses any $x_{1}\in G$
and at the $k$th turn, the concerned player selects $x_{k}\in G\setminus\{x_{1},\ldots,x_{k-1}\}$.
That is, a position in $\gen(G)$ is a set of jointly selected elements.
A player wins on the $n$th turn as soon as $\langle x_{1},\ldots,x_{n}\rangle=G$.

In this paper, we use $\mathbb{Z}_{n}:=\{0,1,\ldots,n-1\}$ to denote
the cyclic group of order $n$ under addition modulo $n$, so that
$\mathbb{Z}_{n}\cong\mathbb{Z}/n\mathbb{Z}$.
\begin{example}
The trivial group $\mathbb{Z}_{1}$ has no maximal subgroups. We cannot
play $\dng(\mathbb{Z}_{1})$ since every subset of the group, including
the empty set, is a generating set. The only position of $\gen(\mathbb{Z}_{1})$
is the empty set, and so the second player wins before the first player
can make a move. This implies that $\gen(\mathbb{Z}_{1})=*0$. 
\end{example}

\begin{example}
Consider the avoidance game on the cyclic group $\mathbb{Z}_{4}$.
No player can choose either 1 or 3 since these elements individually
generate the group. If the first player chooses 0, then the only option
for the second player is 2. After this move, the first player has
no available options. We arrive at the same conclusion if the first
player chooses 2 on the opening move. Regardless, the second player
wins $\dng(\mathbb{Z}_{4})$, which implies that $\dng(\mathbb{Z}_{4})=*0$.
The game digraph for $\dng(\mathbb{Z}_{4})$ is given in Figure~\ref{fig:GENZ4Full}(a).
In the digraph, the vertices are the positions of the game, every
position is connected to its options by arrows, and every position
is labeled by the nimber of the corresponding position.

For the achievement game, it is easy to see that the first player
can win on the opening move by choosing 1 or 3. However, if the first
player happens to choose 0 or 2 on the opening move, then the second
player may choose 1, 3, or the opposite choice that the first player
made on the opening move. The second player wins the game if they
choose either 1 or 3. However, if the second player makes the opposite
choice between 0 and 2 that the first player made, then the first
player wins by choosing either 1 or 3. The game digraph for $\gen(\mathbb{Z}_{4})$
is given in Figure~\ref{fig:GENZ4Full}(b). By looking at the top
node of the digraph, we see that $\gen(\mathbb{Z}_{4})=*1$.
\end{example}
\begin{figure}[h]
\begin{tabular}{ccccc}
\raisebox{1.2cm}{\includegraphics{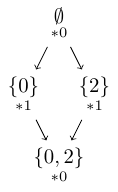}} & $\quad$ & \includegraphics{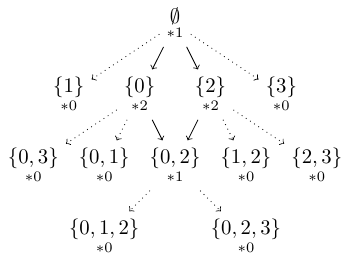} & $\quad$ & \includegraphics{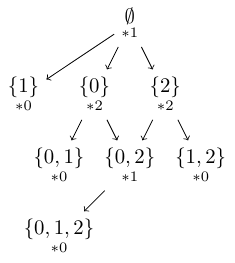}\tabularnewline
(a) $\dng(\mathbb{Z}_{4})$ &  & (b) $\gen(\mathbb{Z}_{4})$ &  & (c) $\gen(\mathbb{Z}_{4})$\tabularnewline
\end{tabular}

\caption{\label{fig:GENZ4Full}Game digraphs of $\protect\dng(\mathbb{Z}_{4})$
and $\protect\gen(\mathbb{Z}_{4})$, and representative game digraph
for $\protect\gen(\mathbb{Z}_{4})$. Nimbers corresponding to each
position of the game are included. The second digraph can be created
from the first by adding the dotted arrows that represent options
that create terminal positions.}
\end{figure}

We call two subsets $P$ and $Q$ of a group $G$ \emph{automorphism
equivalent} if there is an automorphism $\phi$ of $G$ such that
$\phi(P)=Q$. It is clear that an automorphism $\phi$ of $G$ induces
an automorphism of the game digraph, and so $\nim(P)=\nim(Q)$ if
$P$ and $Q$ are automorphism equivalent. For simplicity, we can
eliminate some of the positions of a game digraph without changing
the nim-numbers of the remaining positions. For each position $P$,
the set $\opt(P)$ of options is partitioned into automorphism equivalence
classes. We delete all but one representative from each of these classes.
The resulting digraph will be referred to as a \emph{representative
game digraph.} 
\begin{example}
Consider the achievement game on the cyclic group $\mathbb{Z}_{4}$.
It is clear that the subsets $\{3\}$, $\{0,3\}$, $\{2,3\}$, and
$\{0,2,3\}$ are automorphism equivalent to $\{1\}$, $\{0,1\}$,
$\{1,2\}$, and $\{0,1,3\}$, respectively. As a result, one possible
representative game digraph for $\gen(\mathbb{Z}_{4})$ is provided
in Figure~\ref{fig:GENZ4Full}(c).
\end{example}

We define $\pty(n):=n\bmod2$. The \emph{parity of a subset of a group}
is defined to be the parity of the size of the subset. Observe that
an option of a position in both $\dng$ and $\gen$ has the opposite
parity.

\section{Avoidance games\label{sec:DNG}}

In this section we study the avoidance game $\dng(G)$ on a finite
group $G$. In \cite{anderson.harary:achievement}, Anderson and Harary
proved the following criterion for determining the outcome of $\dng(G)$
for a finite abelian group.
\begin{prop}
\label{prop:Harary}Let $G$ be a finite abelian group. The first
player wins $\dng(G)$ if and only if $G$ is nontrivial of odd order
or $G\cong\mathbb{Z}_{2k}$ with $k$ odd. The second player wins
otherwise.
\end{prop}
Barnes~\cite{Barnes} reproved this result using the following criterion
for determining the outcome of $\dng(G)$ for an arbitrary finite
group $G$. Recall that an involution in a group is an element of
order $2$.
\begin{prop}
\label{prop:Barnes}Let $G$ be a nontrivial finite group. Then the
first player wins $\dng(G)$ if and only if there exists $\alpha\in G$
such that $\alpha$ has odd order and $\langle\alpha,\beta\rangle=G$
for every involution $\beta\in G$.
\end{prop}
Note that the last condition of the previous proposition vacuously
holds if the order of $G$ is odd. Since the first player wins precisely
when $\dng(G)\neq*0$, we immediately have the following corollary
of Proposition~\ref{prop:Harary}.
\begin{cor}
\label{cor:HararyBarnes}Let $G$ be a nontrivial finite abelian group.
Then $\dng(G)\neq*0$ if and only if $G$ has odd order or $G\cong\mathbb{Z}_{2k}$
with $k$ odd.
\end{cor}
The next proposition follows immediately from Proposition~\ref{prop:GeneratingSubset}.
\begin{prop}
The positions of $\dng(G)$ are the subsets of the maximal subgroups
of $G$ and the terminal positions of $\dng(G)$ are the maximal subgroups
of $G$. 
\end{prop}
It turns out that positions that are contained in the same collection
of maximal subgroups are closely related. This motivates the next
two definitions.
\begin{defn}
Let $\mathcal{M}$ be the set of maximal subgroups of $G$. The set
of \emph{intersection subgroups} is defined to be the set 
\[
\mathcal{I}:=\{\cap\mathcal{N}\mid\emptyset\not=\mathcal{N\subseteq\mathcal{M}}\}
\]
containing all the possible intersections of some maximal subgroups. 
\end{defn}
Note that the elements of $\mathcal{I}$ are in fact subgroups of
$G$. If $G$ is nontrivial, then the smallest intersection subgroup
is the Frattini subgroup $\Phi(G)$. Not every subgroup of $G$ is
an intersection subgroup. For example, $\Phi(G)$ may not be trivial.
The set $\mathcal{I}$ of intersection subgroups is partially ordered
by inclusion. We use interval notation to denote certain subsets of
$\mathcal{I}$. For example, if $I\in\mathcal{I}$, then $(-\infty,I):=\{J\in\mathcal{I}\mid J\subset I\}$. 
\begin{defn}
For each $I\in\mathcal{I}$ let
\[
X_{I}:=\mathcal{P}(I)\setminus\cup\{\mathcal{P}(J)\mid J\in(-\infty,I)\}
\]
be the collection of those subsets of $I$ that are not contained
in any other intersection subgroup smaller than $I$. We let $\mathcal{X}:=\{X_{I}\mid I\in\mathcal{I}\}$
and call an element of $\mathcal{X}$ a \emph{structure class}.
\end{defn}
The largest element of $X_{I}$ is $I$. The starting position $\emptyset$
is in $X_{\Phi(G)}$. We say that $X_{I}$ is \emph{terminal} if $I$
is terminal. The \emph{parity of a structure class} is defined to
be $\pty(X_{I}):=\pty(I)$. 
\begin{example}
The subset $P=\{0\}$ generates the trivial subgroup of $G=\mathbb{Z}_{4}$.
If $I=\Phi(G)=\{0,2\}$ is the Frattini subgroup of $G$, then $P\in X_{I}$
since $P\subseteq I$ and $(-\infty,I)$ is empty. This shows that
$P\in X_{I}$ does not imply that $P$ generates $I$. 
\end{example}
\begin{prop}
If $I$ and $J$ are different elements of $\mathcal{I}$, then $X_{I}\cap X_{J}=\emptyset$.
\end{prop}
\begin{proof}
Assume $P\in X_{I}\cap X_{J}$ and let $K:=I\cap J\in\mathcal{I}$.
Since $I\ne J$, we must have $K\ne I$ or $K\ne J$. Without loss
of generality, assume that $K\ne I$. Then $P\subseteq K\in(-\infty,I)$,
which contradicts $P\in X_{I}$.
\end{proof}
\begin{cor}
The set $\mathcal{X}$ of structure classes is a partition of the
set of game positions of $\dng(G)$.
\end{cor}
As we shall see, the structure classes play a pivotal role in the
remainder of this paper. Proposition~\ref{prop:dngCompat} and Corollary~\ref{cor:dngCompat}
imply that the collection of structure classes is compatible with
the option relationship between game positions. This will allow us
to define the structure digraph of $\dng(G)$ (see Definition~\ref{def:structureDigraph}),
which will capture the option relationship among structure classes.
Then in Proposition~\ref{prop:folding}, we show that two elements
of a structure class have the same nim-number if and only if they
have the same parity. In other words, each structure class is associated
with two nim-numbers. By appending this data to the corresponding
structure digraph, we can visualize the nim-number relationship among
structure classes using a structure diagram. By defining an appropriate
equivalence relation on the collection of structure classes (see Definition~\ref{def:simplifiedStructureDiagram}),
we will be able to make identifications that allow us to greatly simplify
the task of computing the nim-number of $\dng(G)$ for a wide class
of groups.
\begin{prop}
\label{prop:dngCompat}Assume $X_{I},X_{J}\in\mathcal{X}$ and $P\in X_{I}\not=X_{J}$.
Then $\opt(P)\cap X_{J}\not=\emptyset$ if and only if $\opt(I)\cap X_{J}\not=\emptyset$.
\end{prop}
\begin{proof}
First, assume that $\opt(P)\cap X_{J}\not=\emptyset$. Then there
exists a $g\in G\setminus P$ such that $P\cup\{g\}\in X_{J}$. That
is, $P\cup\{g\}\subseteq J$ but $P\cup\{g\}$ is not contained in
any $K\in(-\infty,J)$. This implies that $I\subset J$, otherwise
we would have $P\subseteq I\cap J\in(-\infty,I)$, contradicting $P\in X_{I}$.
There is no $K$ satisfying $I\cup\{g\}\subseteq K\in(-\infty,J)$
since we would then have $P\cup\{g\}\subseteq I\cup\{g\}\subseteq K\in(-\infty,J)$,
which contradicts $P\cup\{g\}\in X_{J}$. Thus, $I\cup\{g\}\in X_{J}$,
which shows that $\opt(I)\cap X_{J}\not=\emptyset$.

Now, assume that $\opt(I)\cap X_{J}\not=\emptyset$. Then $I\cup\{g\}\in X_{J}$
for some $g\in J\setminus I$, that is, $I\cup\{g\}\subseteq J$ but
$I\cup\{g\}$ is not contained in any $K\in(-\infty,J)$. Then clearly
$P\cup\{g\}\subseteq J$. There is no $K$ satisfying $P\cup\{g\}\subseteq K\in(-\infty,J)$,
otherwise we would have $P\subseteq K\cap I\in(-\infty,I)$ contradicting
$P\in X_{I}$. Hence $P\cup\{g\}\in X_{J}$, and so $\opt(P)\cap X_{J}\not=\emptyset$,
as desired.
\end{proof}
\begin{cor}
\label{cor:dngCompat}Assume $X_{I},X_{J}\in\mathcal{X}$ and $P,Q\in X_{I}\ne X_{J}$.
Then $\opt(P)\cap X_{J}\not=\emptyset$ if and only if $\opt(Q)\cap X_{J}\ne\emptyset$. 
\end{cor}
This motivates the following definition.
\begin{defn}
\label{def:structureDigraph}We say that $X_{J}$ is an \emph{option}
of $X_{I}$ and we write $X_{J}\in\opt(X_{I})$ if $\opt(I)\cap X_{J}\not=\emptyset$.
The \emph{structure digraph} of $\dng(G)$ has vertex set $\{X_{I}\mid I\in\mathcal{I}\}$
and edge set $\{(X_{I},X_{J})\mid X_{J}\in\opt(X_{I})\}$.
\end{defn}
If $I\not=P\in X_{I}$, then $P\cup\{g\}\in\opt(P)\cap X_{I}$ for
all $g\in I\setminus P\not=\emptyset$. So, $I$ is the only element
of $X_{I}$ without an option in $X_{I}$. Note that there are no
loops in the structure digraph and $X_{\Phi(G)}$ is the only source
vertex. 
\begin{example}
The set of maximal subgroups of $G=\mathbb{Z}_{6}$ is $\mathcal{M}=\{\{0,2,4\},\{0,3\}\}$.
The intersection subgroups are in $\mathcal{I}=\{\{0\},\{0,2,4\},\{0,3\}\}$.
Note that $\Phi(G)=\{0\}$. The elements of $\mathcal{X}$ are 
\[
X_{\{0\}}=\{\emptyset,\{0\}\},\quad X_{\{0,2,4\}}=\{\{2\},\{4\},\{0,2\},\{0,4\},\{2,4\},\{0,2,4\}\},\quad X_{\{0,3\}}=\{\{3\},\{0,3\}\}.
\]
The structure digraph is visualized by
\[
X_{\{0,2,4\}}\longleftarrow X_{\{0\}}\longrightarrow X_{\{0,3\}}.
\]
\end{example}
The following lemma will be useful in the proof of Proposition~\ref{prop:folding}.
\begin{lem}
\label{lem:mex}Let $A$ and $B$ be subsets of $\{0,1,2,\ldots\}$
such that $\mex(A)\in B$. Then $\mex(A\cup\{\mex(B)\})=\mex(A)$.
\end{lem}
\begin{proof}
Since $\mex(A)\in B$, $\mex(B)\not=\mex(A)$. This implies that $\mex(A\cup\{\mex(B)\})=\mex(A)$.
\end{proof}
\begin{figure}
\begin{tabular}{ccc}
\includegraphics{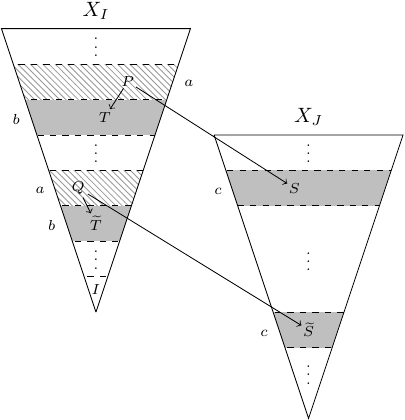} &  & \includegraphics{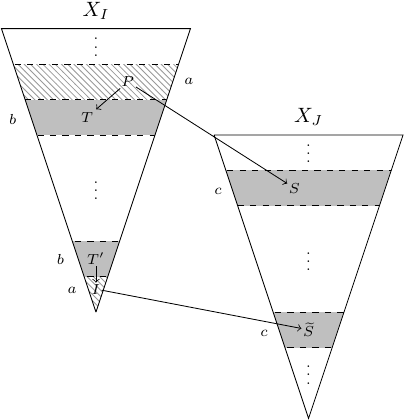}\tabularnewline
(a) $Q\not=I$ & $\qquad$ & (b) $Q=I$\tabularnewline
\end{tabular}

\caption{\label{fig:dngCompat}Unfolded structure diagrams for Proposition~\ref{prop:folding}.
We define $a:=\protect\nim(Q)$, $b:=\protect\nim(T)$ and $c:=\protect\nim(S)$.
Shading types represent parities.}
\end{figure}

\begin{prop}
\label{prop:folding} If $P,Q\in X_{I}\in\mathcal{X}$ and $\pty(P)=\pty(Q)$,
then $\nim(P)=\nim(Q)$.
\end{prop}
\begin{proof}
We proceed by structural induction. Our inductive hypothesis states
that if $\pty(M)=\pty(N)$ and either $M,N\in X_{I}$ with $|M|,|N|>\min\{|P|,|Q|\}$
or $M,N\in X_{J}\in\opt(X_{I})$ for some $J\in\mathcal{I}$, then
$\nim(M)=\nim(N)$. Without loss of generality, assume that $|P|\leq|Q|$.

We will consider two cases indicated by the diagrams in Figure~\ref{fig:dngCompat}:
$Q\neq I$ and $Q=I$. Each figure is referred to as an ``unfolded
structure diagram'' and is meant to help visualize the structure
of the proof. In the figures, each triangle represents a structure
class and each horizontal band represents the collection of subsets
from the given structure class that have the same size. An arrow from
one set to another in the figure indicates that the second set is
an option of the first set. In the first case ($Q\neq I$), we will
show that $\nopt(P)=\nopt(Q)$. The second case ($Q=I$) is more complicated
as we will only have $\nopt(Q)\subseteq\nopt(P)$. In this case, we
will make use of Lemma~\ref{lem:mex} to conclude that we still have
$\mex(\nopt(P))=\mex(\nopt(I))$.

First, assume that $Q\neq I$. Every option of $P$ is either an element
of $X_{I}$ or of some $X_{J}\in\opt(X_{I})$. If $T\in\opt(P)\cap X_{I}$,
then we can choose $\widetilde{T}\in\opt(Q)\cap X_{I}$. By induction,
$\nim(T)=\nim(\widetilde{T})$ since $\pty(T)=\pty(\tilde{T})$. On
the other hand, if $S\in\opt(P)\cap X_{J}$ for some $X_{J}\in\opt(X_{I})$,
then by Corollary~\ref{cor:dngCompat}, there exists $\tilde{S}\in\opt(Q)\cap X_{J}$.
Since $\pty(S)=\pty(\tilde{S})$, we must have $\nim(S)=\nim(\tilde{S})$
by induction. We have shown that $\nopt(P)\subseteq\nopt(Q)$. A similar
argument shows that $\nopt(Q)\subseteq\nopt(P)$.

Now, assume that $Q=I$. If $P=Q=I$, then the desired result follows
trivially, so assume that $P\neq Q=I$. Note that this implies that
$|P|<|Q|$. In this case, we might not have $\nopt(Q)=\nopt(P)$ because
$\opt(I)\cap X_{I}=\emptyset$. Instead we fix a $T^{\prime}\in X_{I}$
such that $I\in\opt(T^{\prime})$. Such a $T'$ exists since $I\not=\emptyset$.
We are going to show that $\nopt(P)=\nopt(I)\cup\{\nim(T')\}$. 

First, we verify $\nopt(P)\subseteq\nopt(I)\cup\{\nim(T')\}$. Every
option of $P$ is either an element of $X_{I}$ or of some $X_{J}\in\opt(X_{I})$.
In the latter case, if $S\in\opt(P)\cap X_{J}$ for some $X_{J}\in\opt(X_{I})$,
then there exists $\tilde{S}\in\opt(Q)\cap X_{J}$ by Corollary~\ref{cor:dngCompat}.
Since $\pty(S)=\pty(\tilde{S})$, we must have $\nim(S)=\nim(\tilde{S})\in\nopt(I)$
by induction. In the former case, if $T\in\opt(P)\cap X_{I}$, then
$\nim(T)=\nim(T')$ by induction and $\nim(T)\in\{\nim(T')\}$.

Now, we verify $\nopt(P)\supseteq\nopt(I)\cup\{\nim(T')\}$. Suppose
$\widetilde{S}\in\opt(I)$. Then $\widetilde{S}\in X_{J}$ for some
$X_{J}\in\opt(X_{I})$ since the only options of $I$ must exist outside
$X_{I}$. Then by Corollary~\ref{cor:dngCompat}, there exists $S\in\opt(P)\cap X_{J}$.
Since $\pty(\widetilde{S})=\pty(S)$, we must have $\nim(\widetilde{S})=\nim(S)$
by induction. This implies that $\nopt(I)\subseteq\nopt(P)$. Since
$P\neq I$, there exits $T\in\opt(P)\cap X_{I}$. By induction, $\nim(T')=\nim(T)$
and so $\{\nim(T')\}\subseteq\nopt(P)$ by induction.

Finally
\[
\begin{aligned}\nim(P) & =\mex(\nopt(P))=\mex(\nopt(I)\cup\{\nim(T')\})\\
 & =\mex(\nopt(I)\cup\{\mex(\nopt(T'))\})\\
 & =\mex(\nopt(I))=\nim(Q)
\end{aligned}
\]
by Lemma~\ref{lem:mex} since $\mex(\nopt(I))=\nim(I)\in\nopt(T')$.
\end{proof}
The upshot of Proposition~\ref{prop:folding} is that each structure
class is associated with two nim-numbers. In light of this, we can
append this information to a structure digraph to form a (folded)
\emph{structure diagram}, which is visualized as follows. A structure
class $X_{I}$ is represented by a triangle pointing down if $I$
is odd and by a triangle pointing up if $I$ is even. The triangles
are divided into a smaller triangle and a trapezoid. The smaller triangle
represents the odd positions of $X_{I}$ and the trapezoid represents
the even positions of $X_{I}$. The numbers are the nim-numbers of
these positions. There is a directed edge from $X_{I}$ to $X_{J}$
provided $X_{J}\in\opt(X_{I})$.

For a nontrivial group $G$, $\Phi(G)$ and $\Phi(G)\setminus\{e\}$
are positions in $X_{\Phi(G)}$. Hence $X_{\Phi(G)}$ contains both
even and odd positions. Corollary~\ref{cor:dngCompat} now implies
that every structure class contains both even and odd positions. The
nim-number of the game can be easily read from the structure diagram.
It is the number in the trapezoid part of the triangle representing
the source vertex of the structure digraph. 

\begin{figure}
\begin{tabular}{ccccc}
\includegraphics{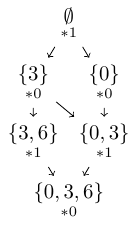} & $\qquad$ & \includegraphics{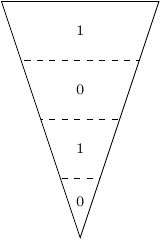} & $\qquad$ & \includegraphics{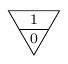}\tabularnewline
(a) &  & (b) &  & (c)\tabularnewline
\end{tabular}\caption{\label{fig:Z9}Representative game digraph, unfolded structure diagram,
and structure diagram for $\protect\dng(\mathbb{Z}_{9})$. }
\end{figure}

\begin{example}
Figure~\ref{fig:Z9} shows a representative game digraph and the
corresponding unfolded and folded structure diagrams for $\dng(\mathbb{Z}_{9})$.
Since $\mathbb{Z}_{9}$ only has a single maximal subgroup, there
is a unique structure class. The small triangle in the structure diagram
represents the collection of odd positions with nim-number $0$. This
collection includes the representative positions $\{0,3,6\},\{3\}$,
and $\{0\}$. The trapezoid represents the collection of even positions
with nim-number $1$. This collection includes positions $\{3,6\},\{0,3\}$,
and $\emptyset$. Every position in a collection is connected to another
position on the next level. The chain shown in the figure ends on
an odd level with the terminal position $I=\{0,3,6\}$. This is why
the large triangle representing $X_{I}$ points down and the small
triangle representing the odd positions is on the bottom. 

The structure diagram can be used to find the nim-numbers. Since the
only terminal position is in the smaller triangle, these positions
have nim-number $0$. The positions in the trapezoid are only connected
to the positions in the smaller triangle so they have nim-number $\mex\{0\}=1$.
The non-terminal positions in the smaller triangle are only connected
to positions in the trapezoid so they have nim-number $\mex\{1\}=0$.
We conclude that $\dng(\mathbb{Z}_{9})=*1$.
\end{example}
We want to recognize similar structure classes. 
\begin{defn}
\label{def:simplifiedStructureDiagram}The \emph{type} of the structure
class $X_{I}$ is the triple 
\[
\type(X_{I}):=(\pty(I),\nim(P),\nim(Q))
\]
 where $P,Q\in X_{I}$ with $\pty(P)=0$ and $\pty(Q)=1$. The \emph{option
type} of $X_{I}$ is the set 
\[
\otype(X_{I}):=\{\type(X_{K})\mid X_{K}\in\opt(X_{I})\},
\]
and the \emph{full option type} of $X_{I}$ is the set
\[
\Otype(X_{I}):=\otype(X_{I})\cup\{\type(X_{I})\}.
\]
We say $X_{I}$ and $X_{J}$ are \emph{type equivalent} if $\type(X_{I})=\type(X_{J})$
and $\Otype(X_{I})=\Otype(X_{J})$. 
\end{defn}
\begin{figure}
\begin{tabular}{ccc}
$X_{I}$ & \raisebox{-4mm}{\includegraphics{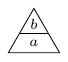}} & \raisebox{-4mm}{\includegraphics{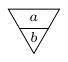}}\tabularnewline
$\type(X_{I})$ & $(0,a,b)$ & $(1,a,b)$\tabularnewline
\end{tabular}\caption{\label{fig:genericTypes}Visualization of structure classes and their
corresponding types.}
\end{figure}

Once the structure digraph is known, the types of the structure classes
can be computed recursively from the bottom up using the formulas
$\type(X_{I})=(\pty(I),a,b)$, where
\[
\begin{gathered}A=\{\tilde{a}\mid(\epsilon,\tilde{a},\tilde{b})\in\otype(X_{I})\},\quad B=\{\tilde{b}\mid(\epsilon,\tilde{a},\tilde{b})\in\otype(X_{I})\},\\
a:=\mex(B),\,b:=\mex(A\cup\{a\})\text{ if }\pty(I)=0\\
b:=\mex(A),\,a:=\mex(B\cup\{b\})\text{ if }\pty(I)=1.
\end{gathered}
\]

Figure~\ref{fig:genericTypes} shows the type-dependent visualization
of structure classes that occur as nodes of a structure diagram. Note
that if $X_{I}$ is terminal, then $\type(X_{I})$ must be either
$(0,0,1)$ or $(1,1,0)$ depending on the parity of $X_{I}$. A structure
digraph can be quite complicated, but we can simplify it by identifying
some vertices.
\begin{defn}
The \emph{simplified structure digraph} of $\dng(G)$ is built from
the structure digraph of $\dng(G)$ by the identification of type
equivalent structure classes followed by the removal of any resulting
loops. The \emph{simplified structure diagram }is built from the structure
diagram using the same identification process. 
\end{defn}
\begin{figure}
\begin{tabular}{cc}
\raise20pt\hbox{\includegraphics{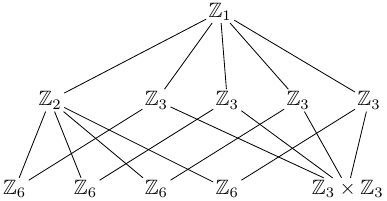}} & \includegraphics{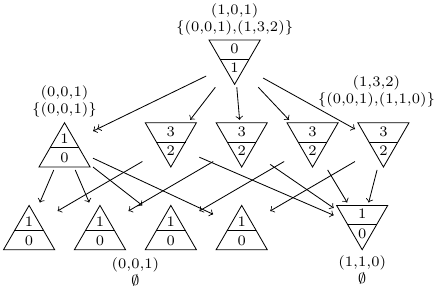}\tabularnewline
(a) Intersection subgroup poset. Subgroups & (b) Structure diagram with\tabularnewline
\hspace{6mm}are denoted by their isomorphism types. & \hspace{-7mm}$\type$ and $\otype$.\tabularnewline
 & \tabularnewline
\includegraphics{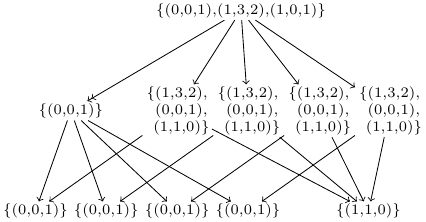} & \includegraphics{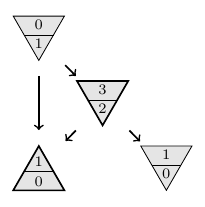}\tabularnewline
(c) The $\Otype$ of each structure class. & (d) Simplified structure diagram.\tabularnewline
\end{tabular}

\caption{\label{fig:49dng}The process for obtaining the simplified structure
diagram for $\protect\dng(\mathbb{Z}_{6}\times\mathbb{Z}_{3})$. The
double headed arrow $X_{\Phi(G)}\twoheadrightarrow X_{\mathbb{Z}_{2}}$
indicates that $X_{\Phi(G)}$ is also connected to the options of
$X_{\mathbb{\mathbb{Z}}_{2}}$. }
 
\end{figure}

\begin{example}
Figure~\ref{fig:49dng} illustrates the main steps of finding the
simplified structure diagram of $\dng(G)$ with $G=\mathbb{Z}_{6}\times\mathbb{Z}_{3}$.
Subfigure~(a) shows the Hasse diagram of the poset of intersection
subgroups. For this group every proper subgroup is an intersection
subgroup. Each intersection subgroup corresponds to a structure class
shown in Subfigure~(b). 

The arrows coming out of a structure class $X_{I}$ are determined
by finding the structure classes containing $I\cup\{g\}$ for $g\in G\setminus I$
according to Definition~\ref{def:structureDigraph}. Next, we compute
$\type(X_{I})$ and $\otype(X_{I})$ for all $X_{I}$. These values
are also shown in Subfigure~(b). For example, $X_{\Phi(G)}$ is the
structure class at the top of the diagram with $\type(X_{\Phi(G)})=(1,0,1)$
and $\otype(X_{\Phi(G)})=\{(0,0,1),(1,3,2)\}$. 

Subfigure~(c) depicts $\Otype(X_{I})$ for all $X_{I}$, which are
computed as the union of $\type(X_{I})$ and $\otype(X_{I})$. Note
that each structure class $X_{I}$ with $\type(X_{I})=(0,0,1)$ have
the same $\Otype(X_{I})$ even though they do not have the same $\otype(X_{I})$. 

The final step is the identification of the structure classes according
to $\type$ and $\Otype$. This results in the simplified structure
diagram of Subfigure~(d). We have shaded the triangles in the simplified
structure diagram to signify that we have identified type equivalent
structure classes. In addition, if more than one structure class has
been identified, the boundary of the corresponding triangle is bold. 
\end{example}
\begin{prop}
The type of a structure class $X_{I}$ of $\dng(G)$ is in 
\[
T=\{t_{1}:=(0,0,1),t_{2}:=(1,0,1),t_{3}:=(1,1,0),t_{4}:=(1,3,2)\}.
\]
\end{prop}
\begin{proof}
Recall that the type of a terminal structure class is either $t_{1}$
or $t_{3}$. We show that if $\otype(X_{I})\subseteq T$, then $\type(X_{I})\in T$,
as well. This implies the statement by structural induction. 

If $X_{J}$ is an option of $X_{I}$, then $I$ is a subgroup of $J$,
and so $\pty(X_{J})=1$ implies $\pty(X_{I})=1$. Hence if $(1,a,b)\in\otype(X_{I})$
for some $a$ and $b$, then $\type(X_{I})$ must be of the form $(1,c,d)$.
The following table shows the possibilities for $\otype(X_{I})$ and
$\type(X_{I})$:

\smallskip\centerline{%
\begin{tabular}{|c|c|c|c|c|c|c|c|}
\cline{1-2} \cline{4-5} \cline{7-8} 
$\otype(X_{I})$ & $\type(X_{I})$ &  & $\otype(X_{I})$ & $\type(X_{I})$ &  & $\otype(X_{I})$ & $\type(X_{I})$\tabularnewline
\cline{1-2} \cline{4-5} \cline{7-8} 
$\{t_{1}\}$ & $t_{1}$ or $t_{2}$ &  & $\{t_{1},t_{3}\}$ & $t_{4}$ &  & $\{t_{1},t_{2},t_{3}\}$ & $t_{4}$\tabularnewline
\cline{1-2} \cline{4-5} \cline{7-8} 
$\{t_{2}\}$ & $t_{2}$ &  & $\{t_{1},t_{4}\}$ & $t_{2}$ &  & $\{t_{1},t_{2},t_{4}\}$ & $t_{2}$\tabularnewline
\cline{1-2} \cline{4-5} \cline{7-8} 
$\{t_{3}\}$ & $t_{3}$ &  & $\{t_{2},t_{3}\}$ & $t_{4}$ &  & $\{t_{1},t_{3},t_{4}\}$ & $t_{4}$\tabularnewline
\cline{1-2} \cline{4-5} \cline{7-8} 
$\{t_{4}\}$ & $t_{3}$ &  & $\{t_{2},t_{4}\}$ & $t_{2}$ &  & $\{t_{2},t_{3},t_{4}\}$ & $t_{4}$\tabularnewline
\cline{1-2} \cline{4-5} \cline{7-8} 
$\{t_{1},t_{2}\}$ & $t_{2}$ &  & $\{t_{3},t_{4}\}$ & $t_{3}$ &  & $\{t_{1},t_{2},t_{3},t_{4}\}$ & $t_{4}$\tabularnewline
\cline{1-2} \cline{4-5} \cline{7-8} 
\end{tabular}$\quad$%
\begin{tabular}{|c|c|}
\hline 
\includegraphics[scale=0.7]{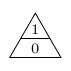} & \includegraphics[scale=0.7]{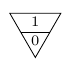}\tabularnewline
$t_{1}$ & $t_{3}$\tabularnewline
\hline 
\includegraphics[scale=0.7]{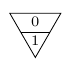} & \includegraphics[scale=0.7]{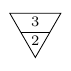}\tabularnewline
$t_{2}$ & $t_{4}$\tabularnewline
\hline 
\end{tabular}}\smallskip

\noindent We show the calculation for the case $\otype(X_{I})=\{t_{4}\}$
which is the fourth case in the first column of the table. Since a
structure class with type $t_{4}=(1,3,2)$ is odd, $X_{I}$ must be
odd, and so $\type(X_{I})=(1,c,d)$ as shown below:

\noindent \begin{center}\includegraphics{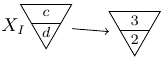}\end{center}

\noindent We know that a position and any of its option have the opposite
parity. So, $d=\mex\{3\}=0$, and hence $c=\mex\{d,2\}=\mex\{0,2\}=1$.
Thus, $\type(X_{I})=(1,1,0)=t_{3}$.

Now we show the calculation for the case $\otype(X_{I})=\{t_{1},t_{4}\}$
which is the second case in the second table. Again since $t_{4}\in\otype(X_{I})$,
$\type(X_{I})$ is of the form $(1,c,d)$. Then $d=\mex\{0,3\}=1$,
and hence $c=\mex\{d,1,2\}=\mex\{1,2\}=0$. Thus, $\type(X_{I})=(1,0,1)=t_{2}$. 

The rest of the calculations in the table are done similarly.
\end{proof}
The next three results strengthen Corollary~\ref{cor:HararyBarnes}.
\begin{cor}
\label{cor:nim-value-dng-013}The nim-number of $\dng(G)$ is $0$,
$1$, or $3$.
\end{cor}
\begin{proof}
The nim-number of $\dng(G)$ is the second digit of $\type(X_{\Phi(G)})$.
\end{proof}
\begin{prop}
\label{prop:DNGodd}If $G$ is nontrivial with odd order, then the
simplified structure diagram of $\dng(G)$ is 

\centerline{\includegraphics{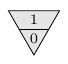}}

\noindent and hence $\dng(G)=*1$.
\end{prop}
\begin{proof}
It is clear that every structure class is odd. Structural induction
on the structure classes shows that $\type(X_{I})=(1,1,0)$ and $\Otype(X_{I})=\{(1,1,0)\}$
for all $X_{I}\in\mathcal{X}$. 
\end{proof}
\begin{prop}
\label{prop:DNGevenFrat}If $G$ has an even Frattini subgroup, then
the simplified structure diagram of $\dng(G)$ is 

\centerline{\includegraphics{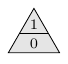}}

\noindent and hence $\dng(G)=*0$.
\end{prop}
\begin{proof}
Every structure class is even since every intersection subgroup contains
the Frattini subgroup. Structural induction on the structure classes
shows that $\type(X_{I})=(0,0,1)$ and $\Otype(X_{I})=\{(0,0,1)\}$
for all $X_{I}\in\mathcal{X}$.
\end{proof}

\section{Achievement games\label{sec:GEN}}

In this section we study the achievement game $\gen(G)$ on the finite
group $G$. For achievement games, we must include an additional structure
class $X_{G}$ containing terminal positions. A subset $S\subseteq G$
belongs to $X_{G}$ whenever $S$ generates $G$ while $S\setminus\{s\}$
does not for some $s\in S$. Note that this is a slightly abusive
notation because, for example, $X_{G}$ does not always contain $G$.
For nontrivial groups, the positions of $\gen(G)$ are the positions
of $\dng(G)$ together with the elements of $X_{G}$. If $G$ is the
trivial group, then $\Phi(G)=G$ is not a game position of $\gen(G)=*0$
and $X_{G}=\{\emptyset\}$ is the only structure class. The following
is immediate. 
\begin{prop}
The set $\mathcal{Y}:=\mathcal{X}\cup\{X_{G}\}$ is a partition of
the set of game positions of $\gen(G)$. 
\end{prop}
We show that the partition $\mathcal{Y}$ is compatible with the option
relationship between positions. The next result is analogous to Proposition~\ref{prop:dngCompat}.
\begin{prop}
\label{prop:genCompat}Assume $X_{I},X_{J}\in\mathcal{Y}$ and $P\in X_{I}\neq X_{J}$
. Then $\opt(P)\cap X_{J}\not=\emptyset$ if and only if $\opt(I)\cap X_{J}\not=\emptyset$. 
\end{prop}
\begin{proof}
It suffices to consider the case when $J=G$ since the proofs of the
other cases are the same as that of Proposition~\ref{prop:dngCompat}.
We show that $\opt(P)\cap X_{G}\not=\emptyset$ for $P\in X_{I}$
if and only if $\opt(I)\cap X_{G}\not=\emptyset$. Note that $\opt(P)\cap X_{G}\not=\emptyset$
when there exists a $g\in G\setminus I$ such that $P\cup\{g\}$ generates
$G$.

First, assume that $\opt(I)\cap X_{G}=\emptyset$. Then $\langle P\cup\{g\}\rangle\subseteq\langle I\cup\{g\}\rangle\not=G$
for all $g\in G\setminus I$ and so $\opt(P)\cap X_{G}=\emptyset$. 

Now, assume that $\opt(I)\cap X_{G}\not=\emptyset$. Then $I\cup\{g\}\in X_{G}$
for some $g\in G\setminus I$, so that $\langle I\cup\{g\}\rangle=G$.
For a contradiction, assume that $\opt(P)\cap X_{G}=\emptyset$. Then
$H:=\langle P\cup\{g\}\rangle$ is a proper subgroup of $G$, and
so $H$ is contained is some maximal subgroup $M$. This implies that
$I$ is not a subset of $M$ since $g\in M$ and $I\cup\{g\}$ generates
$G$. Hence $P\subseteq M\cap I\in(-\infty,I)$, which contradicts
$P\in X_{I}$.
\end{proof}
\begin{cor}
\label{cor:genCompat}Assume $X_{I},X_{J}\in\mathcal{Y}$ and $P,Q\in X_{I}\ne X_{J}$.
Then $\opt(P)\cap X_{J}\not=\emptyset$ if and only if $\opt(Q)\cap X_{J}\not=\emptyset$.
\end{cor}
As with $\dng(G)$, it is convenient to append the nim-number data
for each structure class to the structure digraph for $\gen(G)$,
which we visualize in a structure diagram. The following proposition
guarantees that we may define the simplified structure diagram of
$\gen(G)$ analogously to that of $\dng(G)$. 
\begin{prop}
If $P,Q\in X_{I}\in\mathcal{Y}$ and $\pty(P)=\pty(Q)$, then $\nim(P)=\nim(Q)$.
\end{prop}
\begin{proof}
The proof is analogous to that of Proposition~\ref{prop:folding}
using Corollary~\ref{cor:genCompat}.
\end{proof}
\begin{figure}
\includegraphics{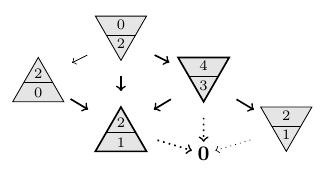}\caption{\label{fig:49gen}Simplified structure diagram for $\protect\gen(\mathbb{Z}_{6}\times\mathbb{Z}_{3})$.}
\end{figure}

Given a structure class $X_{I}$, we define $\type(X_{I})$, $\otype(X_{I})$,
and $\Otype(X_{I})$ as before. Note that by definition the type of
the terminal structure class $X_{G}$ is $\type(X_{G})=(\pty(G),0,0)$.
As in the avoidance games, we define $X_{I}$ and $X_{J}$ to be \emph{type
equivalent} if $\type(X_{I})=\type(X_{J})$ and $\Otype(X_{I})=\Otype(X_{J})$.
The simplified structure digraph of $\gen(G)$ is the identification
digraph of the structure digraph with respect to type equivalence.
The simplified structure diagram is visualized as expected, but the
structure class $X_{G}$ will be denoted by $0$, regardless of the
parity of $X_{G}$.
\begin{example}
Figure~\ref{fig:49gen} shows a simplified structure diagram for
$\gen(\mathbb{Z}_{6}\times\mathbb{Z}_{3})=*0$. Note that in the simplified
structure diagram of $\dng(\mathbb{Z}_{6}\times\mathbb{Z}_{3})$,
shown in Figure~\ref{fig:49dng}(d), all the classes with type $(0,0,1)$
were identified. This is not the case in the simplified structure
digraph of $\gen(\mathbb{Z}_{6}\times\mathbb{Z}_{3})$ since some
of these classes are connected to terminal positions while others
are not. The use of the dotted arrows is to emphasize that these arrows
are not present in $\dng(\mathbb{Z}_{6}\times\mathbb{Z}_{3})$.
\end{example}
\begin{defn}
We call a structure class $X_{I}$ \emph{semi-terminal} if the terminal
class $X_{G}$ is an option of $X_{I}$. We call $X_{I}$ \emph{non-terminal}
if $X_{I}$ is neither terminal nor semi-terminal.
\end{defn}
Note that if $I$ is a maximal subgroup, then $X_{I}$ is semi-terminal.
However, $X_{I}$ may be semi-terminal even if $I$ is not a maximal
subgroup. For example, $X_{\langle(0,1)\rangle}$ is semi-terminal
in $\gen(\mathbb{Z}_{6}\times\mathbb{Z}_{3})$ but $\langle(0,1)\rangle\cong\mathbb{Z}_{3}$
is not a maximal subgroup of $\mathbb{Z}_{6}\times\mathbb{Z}_{3}$,
as shown in Figures~\ref{fig:49dng}(a) and \ref{fig:49gen}. Also
note that a semi-terminal structure class cannot have a non-terminal
option since if $X_{J}$ is an option of $X_{I}$, then $I\le J$. 
\begin{prop}
\label{prop:genTypes}If $G$ is a nontrivial group with $|G|$ odd,
then the type of a structure class $X_{I}$ of $\gen(G)$ is in
\[
T=\{t_{0}:=(1,0,0),t_{1}:=(1,1,0),t_{2}:=(1,2,0),t_{3}:=(1,2,1)\}.
\]
\end{prop}
\begin{proof}
First, it is clear that $\type(X_{I})=t_{0}$ if and only if $X_{I}$
is terminal. Now, we use structural induction to argue that $\type(X_{I})=t_{3}$
if $X_{I}$ is semi-terminal and $\type(X_{I})\in\{t_{1},t_{2}\}$
if $X_{I}$ is non-terminal. 

If $X_{I}$ is semi-terminal, then every option of $X_{I}$ is either
terminal or semi-terminal, and so $\otype(X_{I})=\{t_{0}\}$ or $\otype(X_{I})=\{t_{0},t_{3}\}$
by induction. In both cases the type of $X_{I}$ must be $t_{3}$. 

If $X_{I}$ is non-terminal, then no option of $X_{I}$ has type $t_{0}$
and so $\otype(X_{I})\subseteq\{t_{1},t_{2},t_{3}\}$. In each case,
$\type(X_{I})\in\{t_{1},t_{2}\}$. The following table shows the possibilities
for $\otype(X_{I})$ and $\type(X_{I})$:

\smallskip\centerline{%
\begin{tabular}{|c|c|c|c|c|c|c|c|}
\cline{1-2} \cline{4-5} \cline{7-8} 
$\otype(X_{I})$ & $\type(X_{I})$ &  & $\otype(X_{I})$ & $\type(X_{I})$ &  & $\otype(X_{I})$ & $\type(X_{I})$\tabularnewline
\cline{1-2} \cline{4-5} \cline{7-8} 
$\{t_{1}\}$ & $t_{1}$ &  & $\{t_{1},t_{2}\}$ & $t_{1}$ &  & $\{t_{1},t_{2},t_{3}\}$ & $t_{2}$\tabularnewline
\cline{1-2} \cline{4-5} \cline{7-8} 
$\{t_{2}\}$ & $t_{1}$ &  & $\{t_{1},t_{3}\}$ & $t_{2}$ & \multicolumn{1}{c}{} & \multicolumn{1}{c}{} & \multicolumn{1}{c}{}\tabularnewline
\cline{1-2} \cline{4-5} 
$\{t_{3}\}$ & $t_{2}$ &  & $\{t_{2},t_{3}\}$ & $t_{2}$ & \multicolumn{1}{c}{} & \multicolumn{1}{c}{} & \multicolumn{1}{c}{}\tabularnewline
\cline{1-2} \cline{4-5} 
\end{tabular}$\quad$%
\begin{tabular}{|c|c|c|}
\hline 
\includegraphics[scale=0.7]{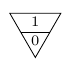} & \includegraphics[scale=0.7]{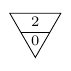} & \includegraphics[scale=0.7]{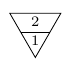}\tabularnewline
$t_{1}$ & $t_{2}$ & $t_{3}$\tabularnewline
\hline 
\multicolumn{1}{c}{} & \multicolumn{1}{c}{} & \multicolumn{1}{c}{}\tabularnewline
\end{tabular}}\smallskip
\end{proof}
\begin{cor}
\label{cor:nimOddGEN12}If $G$ is nontrivial with $|G|$ odd, then
the nim-number of $\gen(G)$ is $1$ or $2$.
\end{cor}
The computer experiments \cite{WEB} of the next section hint at the
following.
\begin{conjecture}
\label{conj:nimEvenGEN01234}If $|G|$ is even, then the nim-number
of $\gen(G)$ is in $\{0,1,2,3,4\}$.
\end{conjecture}
The techniques used to prove Proposition~\ref{prop:genTypes} do
not seem to be sufficient to settle this conjecture. A proof probably
requires a more careful analysis of the forbidden configurations in
a structure diagram. 

\section{Algorithms\label{sec:Algorithms}}

We developed a software package that computes the simplified structure
digraph of $\dng(G)$ and $\gen(G)$. We used GAP \cite{GAP} to get
the maximal subgroups and the rest of the computation is implemented
in C++. The software is efficient enough to allow us to compute the
nim-numbers for the smallest $100,000$ groups which includes all
groups up to size $511$. The result is available on our companion
web page~\cite{WEB}. The algorithms used are shown in Figures~\ref{fig:AlgorithmDNG}
and \ref{fig:AlgorithmGEN}. Both are based on the results of this
section. 
\begin{figure}
\begin{centering}
\begin{tabular}{|c|}
\hline 
\begin{tabular}{l|l|l}
\multicolumn{3}{l}{get the set $\mathcal{M}$ of maximal subgroups of $G$ from GAP //
\emph{terminal positions}}\tabularnewline
\multicolumn{3}{l}{compute the set $\mathcal{I}$ of intersection subgroups}\tabularnewline
\multicolumn{3}{l}{\textbf{for} $I,J\in\mathcal{I}$ // \emph{find arrows}}\tabularnewline
 & \multicolumn{2}{l}{\textbf{if} $I\subseteq J$ and $J\setminus\bigcup\{K\in\mathcal{I}\mid I\subseteq K\text{ and }J\not\subseteq K\}\not=\emptyset$ }\tabularnewline
 &  & add arrow $(X_{I},X_{J})$ to the structure digraph\tabularnewline
\multicolumn{3}{l}{\textbf{for} $I\in\mathcal{I}$ with $\type(X_{I})$ undetermined
// \emph{compute types}}\tabularnewline
 & \multicolumn{2}{l}{\textbf{if }$\type(X_{J})$ is determined for all $X_{J}\in\opt(X_{I})$}\tabularnewline
 &  & compute $\otype(X_{I})$\tabularnewline
 &  & $\type(X_{I}):=(\pty(I),\mex\{c\mid(a,b,c)\in\otype(X_{I})\},\mex\{b\mid(a,b,c)\in\otype(X_{I})\})$\tabularnewline
\multicolumn{3}{l}{\textbf{for} $I,J\in\mathcal{I}$ // \emph{compute the structure classes}}\tabularnewline
 & \multicolumn{2}{l}{\textbf{if} $\type(X_{I})=\type(X_{J})$ and $\Otype(X_{I})=\Otype(X_{J})$ }\tabularnewline
 &  & identify $X_{I}$ and $X_{J}$\tabularnewline
\end{tabular}\tabularnewline
\hline 
\end{tabular}
\par\end{centering}
\caption{\label{fig:AlgorithmDNG}Algorithm to compute the simplified structure
digraph of $\protect\dng(G)$.}
\end{figure}
\begin{figure}
\begin{centering}
\begin{tabular}{|c|}
\hline 
\begin{tabular}{l|l|l}
\multicolumn{3}{l}{get the set $\mathcal{M}$ of maximal subgroups of $G$ from GAP}\tabularnewline
\multicolumn{3}{l}{compute the set $\mathcal{I}$ of intersection subgroups}\tabularnewline
\multicolumn{3}{l}{\textbf{for} $I,J\in\mathcal{I}$ // \emph{find arrows}}\tabularnewline
 & \multicolumn{2}{l}{\textbf{if} $I\subseteq J$ and $J\setminus\cup\{K\in\mathcal{I}\mid I\subseteq K\text{ and }J\not\subseteq K\}\not=\emptyset$ }\tabularnewline
 &  & add arrow $(X_{I},X_{J})$ to the structure digraph\tabularnewline
\multicolumn{3}{l}{\textbf{for} $I\in\mathcal{I}$}\tabularnewline
 & \multicolumn{2}{l}{\textbf{if} $(\exists g\in G)(\forall M\in\mathcal{M})$ $I\cup\{g\}\not\subseteq M$
// $X_{I}$ \emph{is semi-terminal}}\tabularnewline
 &  & add arrow $(X_{I},X_{G})$ to the structure digraph\tabularnewline
\multicolumn{3}{l}{$\type(X_{G}):=(\pty(G),0,0)$ // \emph{the only terminal class }}\tabularnewline
\multicolumn{3}{l}{\textbf{for} $I\in\mathcal{I}$ with $\type(X_{I})$ undetermined }\tabularnewline
 & \multicolumn{2}{l}{\textbf{if }$\type(X_{J})$ is determined for all $X_{J}\in\opt(X_{I})$}\tabularnewline
 &  & compute $\otype(X_{I})$\tabularnewline
 &  & $\type(X_{I}):=(\pty(I),\mex\{c\mid(a,b,c)\in\otype(X_{I})\},\mex\{b\mid(a,b,c)\in\otype(X_{I})\})$\tabularnewline
\multicolumn{3}{l}{\textbf{for} $I,J\in\mathcal{I}$ // \emph{compute the structure classes}}\tabularnewline
 & \multicolumn{2}{l}{\textbf{if} $\type(X_{I})=\type(X_{J})$ and $\Otype(X_{I})=\Otype(X_{J})$ }\tabularnewline
 &  & identify $X_{I}$ and $X_{J}$\tabularnewline
\end{tabular}\tabularnewline
\hline 
\end{tabular}
\par\end{centering}
\caption{\label{fig:AlgorithmGEN}Algorithm to compute the simplified structure
digraph of $\protect\gen(G)$.}
\end{figure}

If $I$ and $J$ are intersection subgroups, it is useful to define
$\mathcal{K}:=\{K\in\mathcal{I}\mid I\subseteq K\text{ and }J\not\subseteq K\}$. 
\begin{lem}
\label{lem:AlgorithmGEN}If $I,J\in\mathcal{I}$ with $I\leq J$,
then $[I,J)=\{J\cap K\mid K\in\mathcal{K}\}$.
\end{lem}
\begin{proof}
Let $L\in[I,J)$. Then $L=\bigcap\mathcal{N}$ with some $\mathcal{N}\subseteq\mathcal{M}$
such that $I\subseteq L\subset J$, and so $L\in\mathcal{K}$. This
implies that $[I,J)\subseteq\mathcal{K}$. Now, suppose $K\in\mathcal{K}$.
Then $K\cap J\in[I,J)$, and hence $\{K\cap J\mid K\in\mathcal{K}\}\subseteq[I,J)$.
Therefore, $[I,J)=\{J\cap K\mid K\in\mathcal{K}\}$.
\end{proof}
The following result is used by the algorithms in Figures~\ref{fig:AlgorithmDNG}
and \ref{fig:AlgorithmGEN}. 
\begin{prop}
Let $I,J\in\mathcal{I}$. We have $X_{J}\in\opt(X_{I})$ if and only
if $I\subseteq J$ and $J\setminus\bigcup\mathcal{K}\not=\emptyset$.
\end{prop}
\begin{proof}
Assume $I\leq J$. First, assume that $X_{J}\in\opt(X_{I})$. Then
there exists $g\in G\setminus I$ such that $I\cup\{g\}\in X_{J}$.
This means there exists $g\in J\setminus I$ such that $I\cup\{g\}\subseteq J$
but $I\cup\{g\}\nsubseteq L$ for any $L\in(-\infty,J)$. Let $K\in\mathcal{K}$,
so that $I\subseteq K$ and $J\nsubseteq K$. Then by Lemma~\ref{lem:AlgorithmGEN},
$J\cap K\in[I,J)\subseteq(-\infty,J)$. Then it must be the case that
$I\cup\{g\}\nsubseteq K$ for all $K\in\mathcal{K}$, which implies
that $g\notin K$ for all $K\in\mathcal{K}$. Hence $g\notin\bigcup\mathcal{K}$.
But since $g\in J$, $J\setminus\bigcup\mathcal{K}\neq\emptyset$.

Now, assume that there exist $g\in J\setminus\bigcup\mathcal{K}$.
Then $I\cup\{g\}\subseteq J$ since $I\leq J$. Moreover, since $g\in J\setminus\bigcup\mathcal{K}$,
$g\notin K$ for all $K\in\mathcal{K}$. But then $g\notin J\cap K$
for all $K\in\mathcal{K}$. Thus, $X_{J}\in\opt(X_{I})$, as desired.
\end{proof}

\section{Cyclic groups\label{sec:Cyclic}}

In this section we study the avoidance game $\dng(G)$ and the achievement
game $\gen(G)$ for a cyclic group $G$. First, we recall some general
results about maximal subgroups.

According to \cite[Theorem 2]{Dlab1960} and \cite[Exercise 7 on page 144]{Suzuki},
we can decompose the Frattini subgroup of a direct product.
\begin{prop}
\label{prop:FrattiniProduct}If $G$ and $H$ are finite groups, then
$\Phi(G\times H)=\Phi(G)\times\Phi(H)$.
\end{prop}
The following result can be found in \cite[Corollary 2]{Dlab1960}
and \cite[Problem 8.1]{Dixon} and is a consequence of Proposition~\ref{prop:FrattiniProduct}.
\begin{prop}
\label{prop:FrattiniCyclic}If $n$ has prime factorization $n=p_{1}^{n_{1}}\cdots p_{k}^{n_{k}}$,
then the Frattini subgroup of $\mathbb{Z}_{n}$ is generated by $p_{1}\cdots p_{k}$
and so it is isomorphic to the cyclic group of order $p_{1}^{n_{1}-1}\cdots p_{k}^{n_{k}-1}$.
\end{prop}
The next result follows from \cite[Exercise 6.1.4 ]{DummitFoote}
and \cite[Problem 140(v)]{Rose}.
\begin{prop}
\label{prop:maximalIndex}A subgroup of a finite abelian group is
maximal if and only if it has prime index.
\end{prop}
Now we are ready to prove our first result about nim-numbers for finite
cyclic groups.

\begin{figure}
\begin{tabular}{ccccc}
\includegraphics{DGraphZ4k} &  & \includegraphics{DGraphZ2k+1} &  & \includegraphics{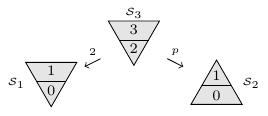}\tabularnewline
(a) $\dng(\mathbb{Z}_{4k})$ &  & (b) $\dng(\mathbb{Z}_{2k+1})$ &  & (c) $\dng(\mathbb{Z}_{4k+2})$\tabularnewline
 &  & $\ \dng(\mathbb{Z}_{2})$ &  & \tabularnewline
\includegraphics{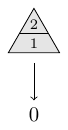} &  & \includegraphics{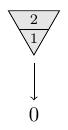} &  & \includegraphics{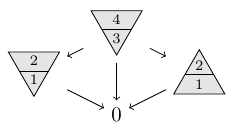}\tabularnewline
(d) $\gen(\mathbb{Z}_{4k})$ &  & (e) $\gen(\mathbb{Z}_{2k+1})$ &  & (f) $\gen(\mathbb{Z}_{4k+2})$\tabularnewline
 &  & $\,\gen(\mathbb{Z}_{2})$ &  & \tabularnewline
\end{tabular}

\caption{\label{fig:structureDiagsCyclic}Simplified structure diagrams for
cyclic groups assuming $k\ge1$. }
\end{figure}

\begin{prop}
\label{prop:DNGZ4k}If $k\ge1$, then the simplified structure diagram
of $\dng(\mathbb{Z}_{4k})$ is the one shown in Figure~\ref{fig:structureDiagsCyclic}(a),
and hence $\dng(\mathbb{Z}_{4k})=*0$.
\end{prop}
\begin{proof}
The result follows from Proposition~\ref{prop:DNGevenFrat} since
$\Phi(\mathbb{Z}_{4k})$ is even by Proposition~\ref{prop:FrattiniCyclic}.
\end{proof}
\begin{prop}
\label{prop:DNGZ4k+2}If $k\ge1$, then the simplified structure diagram
of $\dng(\mathbb{Z}_{4k+2})$ is the one shown in Figure~\ref{fig:structureDiagsCyclic}(c),
and hence $\dng(\mathbb{Z}_{4k+2})=*3$.

\end{prop}
\begin{proof}
Define the following collection of sets that form a partition of $\mathcal{X}$:
\[
\begin{aligned}\mathcal{S}_{1} & :=\{X_{I}\in\mathcal{X}\mid I=\langle2\rangle\};\\
\mathcal{S}_{2} & :=\{X_{I}\in\mathcal{X}\mid X_{I}\text{ is even}\};\\
\mathcal{S}_{3} & :=\{X_{I}\in\mathcal{X}\mid X_{I}\text{ is odd but }I\neq\langle2\rangle\}.
\end{aligned}
\]
We use structural induction on the structure classes to show that
these sets are nonempty and are the type equivalence classes of the
structure classes, and that

\[
\type(X_{I})=\begin{cases}
(1,1,0), & X_{I}\in\mathcal{S}_{1}\\
(0,0,1), & X_{I}\in\mathcal{S}_{2}\\
(1,3,2), & X_{I}\in\mathcal{S}_{3}.
\end{cases}
\]

First, note that $\langle2\rangle$ has order $2k+1$, and hence index
2. By Proposition~\ref{prop:maximalIndex}, $\langle2\rangle$ is
a maximal subgroup, which implies that $X_{\langle2\rangle}\in\mathcal{S}_{1}$
is a terminal structure class. Since $\langle2\rangle$ is odd, $\type(X_{\langle2\rangle})=(1,1,0)$.
Hence $\Otype(X_{\langle2\rangle})=\{(1,1,0)\}$.

Next, let $q$ be an odd prime divisor of $4k+2$. Then $\langle q\rangle$
has index $q$, and so $\langle q\rangle$ is a maximal subgroup by
Proposition~\ref{prop:maximalIndex}. Moreover, $\langle q\rangle$
has even order. It follows that $X_{\langle q\rangle}\in\mathcal{S}_{2}$,
and hence $\mathcal{S}_{2}\neq\emptyset$. Let $X_{I}\in\mathcal{S}_{2}$.
If $X_{I}$ is terminal, then $\type(X_{I})=(0,0,1)$ and $\otype(X_{I})=\emptyset$.
If $X_{I}$ is not terminal, then any option of $X_{I}$ must be even
since $X_{I}$ is even. In this case, $\otype(X_{I})=\{(0,0,1)\}$
by induction. In both cases, $\type(X_{I})=(0,0,1)$, and so $\Otype(X_{I})=\{(0,0,1)\}$.

For the final case, suppose $4k+2$ has prime factorization $2p_{1}^{n_{1}}\cdots p_{r}^{n_{r}}$,
where the $p_{i}$'s are distinct odd primes. By Proposition~\ref{prop:FrattiniCyclic},
$\Phi(\mathbb{Z}_{4k+2})$ is generated by $2p_{1}\cdots p_{r}$ and
is isomorphic to the cyclic group of odd order $p_{1}^{n_{1}-1}\cdots p_{r}^{n_{r}-1}\neq2k+1$.
This implies that $\Phi(\mathbb{Z}_{4k+2})\neq\langle2\rangle$. Hence
$\Phi(\mathbb{Z}_{4k+2})\in\mathcal{S}_{3}$, and so $\mathcal{S}_{3}\neq\emptyset$.\textcolor{red}{{}
}Let $X_{I}\in\mathcal{S}_{3}$ so that $I$ is an odd intersection
subgroup different from $\langle2\rangle$. Since the Frattini subgroup
is a subgroup of $I$, $I=\langle2a\rangle$ for some $a\neq1$ that
divides $p_{1}\cdots p_{r}$. This implies that $\langle I\cup\{2\}\rangle=\langle2\rangle$,
and so $X_{\langle2\rangle}\in\mathcal{S}_{1}$ is an option of $X_{I}$.
Now, let $p$ be a prime divisor of $a$, which is odd. Then $\langle I\cup\{p\}\rangle=\langle p\rangle$
has index $p$, and hence is an even maximal subgroup by Proposition~\ref{prop:maximalIndex}.
This implies that $X_{\langle p\rangle}\in\mathcal{S}_{2}$ is an
option of $X_{I}$. Thus, $\otype(X_{I})$ is either $\{(1,1,0),(0,0,1)\}$
or $\{(1,1,0),(0,0,1),(1,3,2)\}$ by induction. In both cases, $\type(X_{I})=(1,3,2)$.
Therefore, $\Otype(X_{I})=\{(1,1,0),(0,0,1),(1,3,2)\}$. 

It follows that the simplified structure diagram of $\dng(\mathbb{Z}_{4k+2})$
is the one shown in Figure~\ref{fig:structureDiagsCyclic}(c), and
so $\dng(\mathbb{Z}_{4k+2})=*3$. 
\end{proof}
\begin{example}
Figure~\ref{fig:DZ6} shows a representative game digraph and the
simplified structure diagram for $\dng(\mathbb{Z}_{6})$.
\end{example}
\begin{figure}
\begin{tabular}{cccc}
\includegraphics{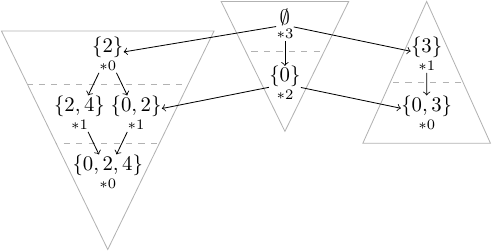} &  &  & \raisebox{1.5cm}{\includegraphics{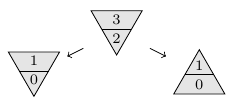}}\tabularnewline
(a) &  &  & (b)\tabularnewline
\end{tabular}

\caption{\label{fig:DZ6}A representative game digraph and the simplified structure
diagram for $\protect\dng(\mathbb{Z}_{6})$.}
\end{figure}

An easy calculation in the $\mathbb{Z}_{2}$ case together with Propositions~\ref{prop:DNGZ4k},
\ref{prop:DNGZ4k+2}, and \ref{prop:DNGodd} immediately yield the
following result.
\begin{cor}
\label{cor:DNGZn}The nim-number of $\dng(\mathbb{Z}_{2})$ is $1$.
If $n\ge3$, then 
\[
\dng(\mathbb{Z}_{n})=\begin{cases}
*1, & n\equiv_{2}1\\
*0, & n\equiv_{4}0\\
*3, & n\equiv_{4}2
\end{cases}.
\]
\end{cor}

\begin{prop}
The nim-number of $\gen(\mathbb{Z}_{n})$ with $n\ge2$ is one larger
than the nim-number of $\dng(\mathbb{Z}_{n})$.
\end{prop}
\begin{proof}
The game digraph of $\gen(\mathbb{Z}_{n})$ is an extension of the
game digraph of $\dng(\mathbb{Z}_{n})$. In this extension every old
vertex in the game digraph of $\dng(\mathbb{Z}_{n})$ gets a new option
that is a generating set since $\mathbb{Z}_{n}$ is cyclic. These
generating sets are terminal positions with nim-number $0$. Figures~\ref{fig:GENZ4Full}(a)
and \ref{fig:GENZ4Full}(b) show this extension for $\mathbb{Z}_{4}$.
After the extension, the nim-number of every old vertex is increased
by one. Figure~\ref{fig:structureDiagsCyclic} shows how the structure
diagrams change during the extension.
\end{proof}
\begin{cor}
The nim-number of $\gen(\mathbb{Z}_{2})$ is $2$. If $n\ge3$, then
\[
\gen(\mathbb{Z}_{n})=\begin{cases}
*2, & n\equiv_{2}1\\
*1, & n\equiv_{4}0\\
*4, & n\equiv_{4}2
\end{cases}.
\]
\end{cor}

\section{Dihedral groups\label{sec:Dihedral}}

In this section we study the avoidance game $\dng(G)$ and the achievement
game $\gen(G)$ for a dihedral group $G$. The following result is
folklore.
\begin{prop}
\label{prop:MaximalDihedral}The subgroups of the dihedral group $\mathbb{D}_{n}=\langle r,f\mid r^{n}=f^{2}=e,rf=fr^{n-1}\rangle$
are either dihedral or cyclic. The maximal subgroups of $\mathbb{D}_{n}$
are the cyclic group $\langle r\rangle$ and the dihedral groups of
the form $\langle r^{p},r^{i}f\rangle\cong\mathbb{D}_{n/p}$ with
prime divisors $p$ of $n$. 
\end{prop}
The following result can be found in \cite[Problem 8.1]{Dixon}.

\begin{prop}
\label{prop:FrattiniDihedral}If $n$ has prime factorization $n=p_{1}^{n_{1}}\cdots p_{k}^{n_{k}}$,
then the Frattini subgroup of $\mathbb{D}_{n}$ is a cyclic group
of order $p_{1}^{n_{1}-1}\cdots p_{k}^{n_{k}-1}$.
\end{prop}
\begin{figure}
\begin{tabular}{ccccc}
\includegraphics{DGraphZ4k} &  & \includegraphics{DGraphd2k+1} &  & \includegraphics{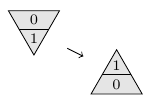}\tabularnewline
(a) $\dng(\mathbb{D}_{4k})$ &  & (b) $\dng(\mathbb{D}_{2k+1})$ &  & (c) $\dng(\mathbb{D}_{4k+2})$\tabularnewline
 &  &  &  & \tabularnewline
\includegraphics{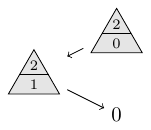} &  & \includegraphics{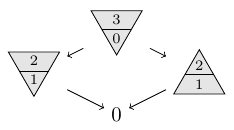} &  & \includegraphics{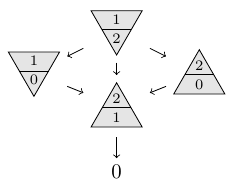}\tabularnewline
(d) $\gen(\mathbb{D}_{4k})$ &  & (e) $\gen(\mathbb{D}_{2k+1})$ &  & (f) $\gen(\mathbb{D}_{4k+2})$\tabularnewline
\end{tabular}

\caption{\label{fig:structureDiagsDihedral}Simplified structure diagrams for
the dihedral groups assuming $k\ge1$.}
\end{figure}

\begin{prop}
If $k\ge1$, then the simplified structure diagrams of $\dng(\mathbb{D}_{4k})$
and $\dng(\mathbb{D}_{4k+2})$ are the ones shown in Figures~\ref{fig:structureDiagsDihedral}(a)
and \ref{fig:structureDiagsDihedral}(c), respectively, and hence
$\dng(\mathbb{D}_{2k+2})=*0$.
\end{prop}

\begin{proof}
The Frattini subgroup of $\mathbb{D}_{4k}$ is even by Proposition~\ref{prop:FrattiniDihedral}.
Thus, the simplified structure diagram of $\dng(\mathbb{D}_{4k})$
is the one depicted in Figure~\ref{fig:structureDiagsDihedral}(a)
by Proposition~\ref{prop:DNGevenFrat}.

The terminal structure classes of $\dng(\mathbb{D}_{4k+2})$ are even
since the maximal subgroups are even by Proposition~\ref{prop:MaximalDihedral}.
On the other hand, the Frattini subgroup is odd by Proposition~\ref{prop:FrattiniDihedral}.
Structural induction on the structure classes shows that $\type(X_{I})=(0,0,1)$
and $\Otype(X_{I})=\{(0,0,1)\}$ if $X_{I}$ is even, while $\type(X_{I})=(1,0,1)$
and $\Otype(X_{I})=\{(1,0,1)\}$ if $X_{I}$ is odd. Every position
that is not terminal has a terminal option since by Proposition~\ref{prop:MaximalDihedral}
we can add an appropriate rotation to the position that creates a
terminal position. Therefore, the simplified structure diagram of
$\dng(\mathbb{D}_{4k+2})$ is the one depicted in Figure~\ref{fig:structureDiagsDihedral}(c).
\end{proof}
\begin{prop}
If $k\ge1$, then the simplified structure diagram of $\dng(\mathbb{D}_{2k+1})$
is the one shown in Figure~\ref{fig:structureDiagsDihedral}(b),
and hence $\dng(\mathbb{D}_{2k+1})=*3$. 
\end{prop}
\begin{proof}
This argument is essentially the same as the one for $\dng(\mathbb{Z}_{4k+2})$
(see the proof of Proposition~\ref{prop:DNGZ4k+2}), where we replace
$2$ with $r$. The only maximal subgroup with odd order is $\langle r\rangle$.
Every other maximal subgroup has even order. 
\end{proof}
\begin{cor}
For $n\ge3$, we have
\[
\dng(\mathbb{D}_{n})=\begin{cases}
*3, & n\equiv_{2}1\\
*0, & n\equiv_{2}0
\end{cases}.
\]
\end{cor}
The outcome of $\gen(\mathbb{D}_{4k})$ can be determined using \cite[Section 3.3]{Barnes}.
We provide information about the structure diagram of the game. 
\begin{prop}
If $k\ge1$, then the simplified structure diagram of $\gen(\mathbb{D}_{4k})$
is the one shown in Figure~\ref{fig:structureDiagsDihedral}(d),
and hence $\gen(\mathbb{D}_{4k})=*0$.
\end{prop}
\begin{proof}
Define the following collection of sets that form a partition of $\mathcal{Y}$:
\[
\begin{aligned}\mathcal{S}_{1} & :=\{X_{I}\in\mathcal{Y}\mid X_{I}\text{ is terminal}\};\\
\mathcal{S}_{2} & :=\{X_{I}\in\mathcal{Y}\mid X_{I}\text{ is semi-terminal}\};\\
\mathcal{S}_{3} & :=\{X_{I}\in\mathcal{Y}\mid X_{I}\text{ is non-terminal}\}.
\end{aligned}
\]
We use structural induction on the structure classes to show that
these sets are nonempty and are the type equivalence classes of the
structure classes, and that
\[
\type(X_{I})=\begin{cases}
(0,0,0), & X_{I}\in\mathcal{S}_{1}\\
(0,1,2), & X_{I}\in\mathcal{S}_{2}\\
(0,0,2), & X_{I}\in\mathcal{S}_{3}
\end{cases}.
\]

First, observe that the Frattini subgroup has even order by Proposition~\ref{prop:FrattiniDihedral}.
This implies that every structure class is even, as well. 

It is clear that $\mathcal{S}_{1}=\{X_{G}\}\neq\emptyset$ and $\type(X_{G})=(0,0,0)$. 

Consider the even maximal subgroup $R=\langle r\rangle$ of $\mathbb{D}_{4k}$.
Then $X_{R}\in\mathcal{S}_{2}$ since $\langle R\cup\{f\}\rangle=\mathbb{D}_{4k}$,
and so $\mathcal{S}_{2}\neq\emptyset$. If $X_{I}\in\mathcal{S}_{2}$,
then $X_{I}$ is semi-terminal, and so $\otype(X_{I})$ is either
$\{(0,0,0)\}$ or $\{(0,0,0),(0,1,2)\}$ by induction. In either case,
$\type(X_{I})=(0,1,2)$, and so $\Otype(X_{I})=\{(0,0,0),(0,1,2)\}$.

For the final case, observe that $\mathcal{S}_{3}\neq\emptyset$ since
the empty position is non-terminal and belongs to $X_{\Phi(\mathbb{D}_{4k})}$.
If $X_{I}\in\mathcal{S}_{3}$, then every option of $X_{I}$ must
be semi-terminal or non-terminal, and so $\otype(X_{I})$ is either
$\{(0,0,2)\}$ or $\{(0,0,2),(0,1,2)\}$ by induction. In either case,
$\type(X_{I})=(0,0,2)$, and hence $\Otype(X_{I})=\{(0,0,2),(0,1,2)\}$.

It follows that the simplified structure diagram of $\gen(\mathbb{D}_{4k})$
is the one shown in Figure~\ref{fig:structureDiagsDihedral}(d),
and so $\gen(\mathbb{D}_{4k})=*0$.
\end{proof}
\begin{prop}
If $k\ge1$, then the simplified structure diagram of $\gen(\mathbb{D}_{2k+1})$
is the one shown in Figure~\ref{fig:structureDiagsDihedral}(e),
and hence $\gen(\mathbb{D}_{2k+1})=*3$.
\end{prop}
\begin{proof}
Let $G:=\mathbb{D}_{2k+1}$. Define the following collection of sets
that form a partition of $\mathcal{Y}$:
\[
\begin{aligned}\mathcal{S}_{1} & :=\{X_{I}\in\mathcal{Y}\mid X_{I}\text{ is terminal}\};\\
\mathcal{S}_{2} & :=\{X_{I}\in\mathcal{Y}\mid X_{I}\text{ is odd and semi-terminal}\};\\
\mathcal{S}_{3} & :=\{X_{I}\in\mathcal{Y}\mid X_{I}\text{ is even and semi-terminal}\};\\
\mathcal{S}_{4} & :=\{X_{I}\in\mathcal{Y}\mid X_{I}\text{ is non-terminal}\}.
\end{aligned}
\]
We use structural induction on the structure classes to show that
these sets are nonempty and are the type equivalence classes of the
structure classes, and that
\[
\type(X_{I})=\begin{cases}
(0,0,0), & X_{I}\in\mathcal{S}_{1}\\
(1,2,1), & X_{I}\in\mathcal{S}_{2}\\
(0,1,2), & X_{I}\in\mathcal{S}_{3}\\
(1,3,0), & X_{I}\in\mathcal{S}_{4}
\end{cases}.
\]

It is clear that $\mathcal{S}_{1}=\{X_{G}\}\neq\emptyset$ and $\type(X_{G})=(0,0,0)$. 

Next, consider the maximal subgroup $R=\langle r\rangle$ of $G$.
The only odd maximal subgroup is $R$ by Proposition~\ref{prop:MaximalDihedral},
and so $\mathcal{S}_{2}=\{X_{R}\}$. Any $T\in\opt(R)$ generates
$G$ since $R$ is a maximal subgroup. So $R$ has no option in a
semi-terminal structure class, and hence $\otype(X_{R})=\{(0,0,0)\}$.
Thus, $\type(X_{R})=(1,2,1)$, and so $\Otype(X_{R})=\{(0,0,0),(1,2,1)\}$.

For the third case, consider the even subgroup $F=\langle f\rangle$.
Then $F$ is a position of an even semi-terminal structure class since
$\langle F\cup\{r\}\rangle=G$, and so $\mathcal{S}_{3}\neq\emptyset$.
Suppose $X_{I}\in\mathcal{S}_{3}$. Since $X_{I}$ is semi-terminal,
$\otype(X_{I})$ is either $\{(0,0,0)\}$ or $\{(0,0,0),(0,1,2)\}$
by induction. In either case, $\type(X_{I})=(0,1,2)$, and so $\Otype(X_{I})=\{(0,0,0),(0,1,2)\}$.

For the final case, note that $\Phi(G)$ is a proper subgroup of $\langle r\rangle$
by Proposition~\ref{prop:FrattiniDihedral}. If $g$ is a rotation
in $G$, then $\Phi(G)\cup\{g\}$ generates a subgroup of $\langle r\rangle$.
If $g$ is a reflection in $G$, then $\Phi(G)\cup\{g\}$ generates
a subgroup $H$ of $G$. It is easy to see that $H$ is isomorphic
to a dihedral group whose order is twice the order of $\Phi(G)$.
Thus $H\ne G$ which means $X_{\Phi(G)}\in\mathcal{S}_{4}$, and so
$\mathcal{S}_{4}\neq\emptyset$. Let $X_{I}\in\mathcal{S}_{4}$. Since
$X_{I}$ is non-terminal, $I$ does not contain $r$ or any reflections.
Then $X_{I}$ must be odd by Cauchy's Theorem because there are no
even order rotations in $G$. Hence $\langle I\cup\{r\}\rangle=\langle r\rangle$
is an option of $X_{I}$ in $\mathcal{S}_{2}$ and $\langle I\cup\{f\}\rangle$
is an option of $X_{I}$ in $\mathcal{S}_{3}$. So $\otype(X_{I})$
is either $\{(1,2,1),(0,1,2)\}$ or $\{(1,2,1),(0,1,2),(1,3,0)\}$
by induction. In either case, $\type(X_{I})=(1,3,0)$, and hence $\Otype(X_{I})=\{(1,2,1),(0,1,2),(1,3,0)\}$.

It follows that the simplified structure diagram of $\gen(G)$ is
the one shown in Figure~\ref{fig:structureDiagsDihedral}(e), and
so $\gen(G)=*3$. 
\end{proof}
\begin{example}
Figure~\ref{fig:DigraphGenD3} shows a representative game digraph
for $\gen(\mathbb{D}_{3})$ and Figure~\ref{fig:structureDiagsDihedral}(e)
shows the simplified structure diagram for $\gen(\mathbb{D}_{3})$.
\end{example}
\begin{figure}
\begin{tabular}{c}
\includegraphics{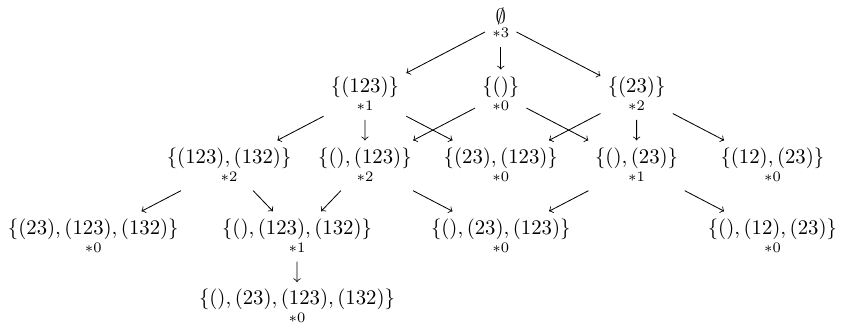}\tabularnewline
\end{tabular}\caption{\label{fig:DigraphGenD3}Representative game digraph of $\protect\gen(\mathbb{D}_{3})$.
We use permutation notation after the identification of $\mathbb{D}_{3}$
with $S_{3}$.}
\end{figure}

\begin{prop}
If $k\ge1$, then the simplified structure diagram of $\gen(\mathbb{D}_{4k+2})$
is the one shown in Figure~\ref{fig:structureDiagsDihedral}(f),
and hence $\gen(\mathbb{D}_{4k+2})=*1$.
\end{prop}
\begin{proof}
Let $G:=\mathbb{D}_{4k+2}$. Define the following collection of sets
that form a partition of $\mathcal{Y}$:
\[
\begin{aligned}\mathcal{S}_{1} & :=\{X_{I}\in\mathcal{Y}\mid X_{I}\text{ is terminal}\};\\
\mathcal{S}_{2} & :=\{X_{I}\in\mathcal{Y}\mid X_{I}\text{ is semi-terminal}\};\\
\mathcal{S}_{3} & :=\{X_{I}\in\mathcal{Y}\mid X_{I}\text{ is even and non-terminal}\};\\
\mathcal{S}_{4} & :=\{X_{I}\in\mathcal{Y}\mid X_{I}\text{ is odd and non-terminal without any even non-terminal option}\};\\
\mathcal{S}_{5} & :=\{X_{I}\in\mathcal{Y}\mid X_{I}\text{ is odd and non-terminal with an even non-terminal option}\}.
\end{aligned}
\]
We use structural induction on the structure classes to show that
these sets are nonempty and are the type equivalence classes of the
structure classes and that
\[
\type(X_{I})=\begin{cases}
(0,0,0), & X_{I}\in\mathcal{S}_{1}\\
(0,1,2), & X_{I}\in\mathcal{S}_{2}\\
(0,0,2), & X_{I}\in\mathcal{S}_{3}\\
(1,1,0), & X_{I}\in\mathcal{S}_{4}\\
(1,1,2), & X_{I}\in\mathcal{S}_{5}
\end{cases}.
\]

As usual, we have $\mathcal{S}_{1}=\{X_{G}\}\ne\emptyset$ and $\type(X_{G})=(0,0,0)$.

Next, consider the even maximal subgroup $R:=\langle r\rangle$. Then
$X_{R}\in\mathcal{S}_{2}$ since $\langle R\cup\{f\}\rangle=G$, and
so $\mathcal{S}_{2}\neq\emptyset$. Suppose $X_{I}\in\mathcal{S}_{2}$.
Then $I=R$ or $I$ contains a reflection, and so $X_{I}$ must be
even. Since $X_{I}$ is semi-terminal, $\otype(X_{I})$ is one of
$\{(0,0,0)\}$ and $\{(0,0,0),(0,1,2)\}$ by induction. In either
case, $\type(X_{I})=(0,1,2)$, and so $\Otype(X_{I})=\{(0,0,0),(0,1,2)\}$.

For the third case, consider the even subgroup $Q=\langle r^{2k+1}\rangle=\{e,r^{2k+1}\}$.
Then $Q\in\mathcal{S}_{3}$ since for all $h\in G$, $\langle Q\cup\{h\}\rangle\neq G$,
and so $\mathcal{S}_{3}\neq\emptyset$. Let $X_{I}\in\mathcal{S}_{3}$.
Then every option of $X_{I}$ is even. Moreover, it must be the case
that $\langle I\cup\{x\}\rangle$ is an option in a structure class
belonging to $\mathcal{S}_{2}$ exactly when $x$ is a reflection\textcolor{red}{{}
}\textcolor{black}{or $x$ is a rotation such that $\langle I\cup\{x\}\rangle$
is the full rotation subgroup.} Otherwise, $\langle I\cup\{x\}\rangle$
is an option in a structure class belonging to $\mathcal{S}_{3}$.
This shows that $\otype(X_{I})$ is either $\{(0,1,2)\}$ or $\{(0,1,2),(0,0,2)\}$
by induction. In either case, $\type(X_{I})=(0,0,2)$, and hence $\Otype(X_{I})=\{(0,1,2),(0,0,2)\}$.

To show that $\mathcal{S}_{4}\neq\emptyset$, consider the odd subgroup
$P=\langle r^{2}\rangle$ and let $x$ be any even order element of
$G$. Then either $x$ is a reflection or equal to $r^{m}$, where
$m$ is odd. In either case, $\langle P\cup\{x\}\rangle$ is maximal,
and hence $\langle P\cup\{x\}\rangle$ must be an element of a semi-terminal
structure class in $\mathcal{S}_{2}$. Hence $P\in\mathcal{S}_{4}$.
Now, let $X_{I}\in\mathcal{S}_{4}$, so that $X_{I}$ is an odd non-terminal
structure class without any even non-terminal options. We see that
$\langle I\cup\{r\}\rangle$ contains the maximal subgroup $\langle r\rangle$,
which has even order. So $X_{I}$ has an option in $\mathcal{S}_{2}$.
Next, we show that $X_{I}$ has no option in $\mathcal{S}_{5}$. For
a contradiction, suppose $X_{I}$ has an option $X_{J}$ in $\mathcal{S}_{5}$.
By the definition of $\mathcal{S}_{5}$, $X_{J}$ has an option $X_{H}$
in $\mathcal{S}_{3}$. Let $t$ be an element of order 2 in $H$ and
let $I\cup\{t\}\in X_{K}$. Then $K$ is even and $K\leq H\leq G$.
So $X_{K}$ is an option of $X_{I}$ in $\mathcal{S}_{3}$, which
is a contradiction. This shows that $\otype(X_{I})$ is either $\{(0,1,2)\}$
or $\{(0,1,2),(1,1,0)\}$ by induction. In either case, $\otype(X_{I})=\{(1,1,0)\}$,
and hence $\Otype(X_{I})=\{(0,1,2),(1,1,0)\}$.

For the final case, consider the non-terminal structure class $X_{\Phi(G)}$,
which contains the empty position $\emptyset$ and is odd by Proposition~\ref{prop:FrattiniDihedral}.
We see that $\langle\emptyset\cup\{r^{2k+1}\}\rangle=\{e,r^{2k+1}\}$
is an option in a structure class belonging to $\mathcal{S}_{3}$.
This shows that $X_{\Phi(G)}\in\mathcal{S}_{5}$, and so $\mathcal{S}_{5}\ne\emptyset$.
Suppose $X_{I}\in\mathcal{S}_{5}$. Since $I$ is of odd order, it
must be a subgroup of $\langle r^{2}\rangle$. However, since $\langle r^{2}\rangle$
is an element of a structure class belonging to $\mathcal{S}_{4}$,
$I$ must be a proper subgroup of $\langle r^{2}\rangle$. We see
that $\langle I\cup\{r\}\rangle=\langle r\rangle$ and $\langle I\cup\{r^{2}\}\rangle=\langle r^{2}\rangle$,
which are elements of structure classes belonging to $\mathcal{S}_{2}$
and $\mathcal{S}_{4}$, respectively. As a consequence, $\otype(X_{I})$
is either $\{(0,0,2),(1,1,0),(0,1,2)\}$ or $\{(0,0,2),(1,1,0),(0,1,2),(1,1,2)\}$
by induction. In either case, $\type(X_{I})=(1,1,2)$, and hence $\Otype(X_{I})=\{(0,0,2),(1,1,0),(0,1,2),(1,1,2)\}$.

It follows that the simplified structure diagram of $\gen(G)$ is
the one shown in Figure~\ref{fig:structureDiagsDihedral}(f), which
implies that $\gen(G)=*1$.  
\end{proof}
\begin{cor}
For $n\ge3$, we have
\[
\gen(\mathbb{D}_{n})=\begin{cases}
*3, & n\equiv_{2}1\\
*0, & n\equiv_{4}0\\
*1, & n\equiv_{4}2
\end{cases}.
\]
\end{cor}
The simplified structure diagrams for $\gen(\mathbb{D}_{n})$ are
shown in Figure~\ref{fig:structureDiagsDihedral}.

\section{Abelian groups\label{sec:Abelian}}

In this section, we study the avoidance game $\dng(G)$ and the achievement
game $\gen(G)$ for a finite abelian group $G$. The following result
is from \cite[Corollary on page 141]{Suzuki}.
\begin{prop}
\label{prop:SubgroupRelPrime}If $G$ and $H$ are finite groups of
relatively prime orders, then any subgroup of $G\times H$ is of the
form $K\times L$ for some subgroups $K\le G$ and $L\le H$.
\end{prop}
The next result completely characterizes the nim-numbers for $\dng(G)$
for finite abelian groups $G$.
\begin{prop}
If $G$ is a finite abelian group, then
\[
\dng(G)=\begin{cases}
*1, & G\text{ is nontrivial of odd order}\\
*1, & G\cong\mathbb{Z}_{2}\\
*3, & G\cong\mathbb{Z}_{2}\times\mathbb{Z}_{2k+1}\text{ with }k\ge1\\
*0, & \text{else}
\end{cases}.
\]
\end{prop}
\begin{proof}
The case when $G$ is nontrivial of odd order is proved in Proposition~\ref{prop:DNGodd}.
Corollary~\ref{cor:DNGZn} covers both the $G\cong\mathbb{Z}_{2}$
and $G\cong\mathbb{Z}_{2}\times\mathbb{Z}_{2k+1}$ cases. In every
other case, $\dng(G)=*0$ since the second player has a winning strategy
as was shown in \cite[Section 3]{anderson.harary:achievement} and
in \cite[Section 2.1]{Barnes} (see Proposition~\ref{prop:Harary}).
\end{proof}
The remainder of this section tackles $\gen(G)$ for finite abelian
groups $G$. Our analysis involving groups with non-zero nim-numbers
handles three cases, which are addressed in Propositions~\ref{prop:abelianGEN-oddPrimes},
\ref{prop:abelianGEN-mkRelPrime}, and \ref{prop:abelianGEN-mkNotRelPrime}. 

Recall that every finite abelian group $G$ has an invariant factor
decomposition $G\cong\mathbb{Z}_{\alpha_{1}}\times\cdots\times\mathbb{Z}_{\alpha_{k}}$
where the positive non-unit elementary divisors $\alpha_{1},\ldots,\alpha_{k}$
are uniquely determined by $G$ and satisfy $\alpha_{i}\mid\alpha_{i+1}$
for $i\in\{1,\ldots,k-1\}$. Our proof of Proposition~\ref{prop:abelianGEN-oddPrimes}
requires the following notion and the lemmas that follow.
\begin{defn}
We define the \emph{spread} $\spr(G)$ of the finite abelian group
$G$ to be the number of elementary divisors in the invariant factor
decomposition of $G$. 
\end{defn}
Note that the trivial group has spread $\spr(\mathbb{Z}_{1})=0$.
If $G=\mathbb{Z}_{p_{1}^{r_{1}}}\times\cdots\times\mathbb{Z}_{p_{k}^{r_{k}}}$
is an abelian group with primes $p_{1},\ldots,p_{k}$, then $\spr(G)=\max\{n_{i}\mid i\in\{1,\ldots,k\}\}$
where $n_{i}:=|\{j\mid p_{i}=p_{j}\}|$. In this case, $G$ is isomorphic
to the direct product of $\spr(G)$-many cyclic groups, but it is
not isomorphic to the direct product of fewer cyclic groups. It follows
that the spread of $G$ is the minimum size of a generating set of
$G$.
\begin{example}
Let $G=\mathbb{Z}_{3}\times\mathbb{Z}_{9}\times\mathbb{Z}_{5}\times\mathbb{Z}_{49}\times\mathbb{Z}_{7}$.
Then $\spr(G)=2$. If $g=(1,1,0,1,0)$, then $G/\langle g\rangle\cong\mathbb{Z}_{3}\times\mathbb{Z}_{5}\times\mathbb{Z}_{7}$,
and so $\spr(G/\langle g\rangle)=1=\spr(G)-1$. If $h=(0,3,1,1,0)$,
then $G/\langle h\rangle\cong\mathbb{Z}_{3}\times\mathbb{Z}_{3}\times\mathbb{Z}_{7}$,
and so $\spr(G/\langle h\rangle)=2=\spr(G)$. If $k=(1,3,1,0,1)$,
then $G/\langle k\rangle\cong\mathbb{Z}_{9}\times\mathbb{Z}_{49}$,
and so $\spr(G/\langle k\rangle)=1=\spr(G)-1$.
\end{example}
The following is an easy consequence of the definitions.
\begin{lem}
\label{lem:lnf}Let $G=\mathbb{Z}_{p_{1}^{r_{1}}}\times\cdots\times\mathbb{Z}_{p_{k}^{r_{k}}}$
be an abelian group with primes $p_{1},\ldots,p_{k}$. If $g=(g_{1},\ldots,g_{k})$
with
\[
g_{i}:=\begin{cases}
1, & i=\min\{j\mid p_{i}=p_{j}\}\\
0, & \text{else}
\end{cases},
\]
then $\spr(G/\langle g\rangle)=\spr(G)-1$.
\end{lem}
\begin{defn}
If $A$ is an $m\times n$ integer matrix, then we define $\ag(A)$
to be the abelian group with generators $g_{1},\ldots,g_{n}$ satisfying
the relations
\[
A\left[\begin{matrix}g_{1}\\
\vdots\\
g_{n}
\end{matrix}\right]=\left[\begin{matrix}0\\
\vdots\\
0
\end{matrix}\right].
\]

It is well-known that $\ag(A)\cong\ag(\snf(A))$ where $\snf(A)$
is the Smith Normal Form of $A$. The diagonal entries $\alpha_{1},\ldots,\alpha_{k}$
of $\snf(A)$ can be calculated as $\alpha_{1}=d_{1}(A)$ and $\alpha_{i}=d_{i}(A)/d_{i-1}(A)$
for $i\in\{2,\ldots,k\}$, where $d_{i}(A)$ is the greatest common
divisor of all $i\times i$ minors of $A$ \cite[Theorem 3.9]{Jacobson}.
Elementary integer row and column operations can be used on $A$ to
get $\snf(A)$. The non-unit diagonal elements of $\snf(A)$ are the
invariant factors of $\ag(A)$ and so $\spr(\ag(A))$ is the number
of non-unit elements of $\{\alpha_{1},\ldots,\alpha_{k}\}$.
\end{defn}
The \emph{truncated Smith Normal Form} $\tsnf(A)$ of $A$ is created
from $\snf(A)$ by removing every zero row and every column containing
a single nonzero entry equal to 1. We clearly have $\ag(A)\cong\ag(\tsnf(A))$.
It is also clear that $\spr(\ag(A))$ is equal to the size of $\tsnf(A)$. 
\begin{example}
Let $G=\mathbb{Z}_{3}\times\mathbb{Z}_{15}\cong\ag\left(\left[\begin{smallmatrix}3 & 0\\
0 & 15
\end{smallmatrix}\right]\right)$. Then $\spr(G)=2$. Now, let $g=(1,2)\in G$ and $A=\left[\begin{smallmatrix}3 & 0\\
0 & 15\\
1 & 2
\end{smallmatrix}\right]$. Then $\snf(A)=\left[\begin{smallmatrix}1 & 0\\
0 & 3\\
0 & 0
\end{smallmatrix}\right]$ and $\tsnf(A)=\left[\begin{smallmatrix}3\end{smallmatrix}\right]$,
and so $G/\langle g\rangle\cong\ag(A)\cong\ag(\snf(A))\cong\mathbb{Z}_{3}$.
Hence $\spr(G/\langle g\rangle)=1=\spr(G)-1$.
\end{example}

\begin{defn}
Let $\mathbb{Z}_{r_{1}}\times\cdots\times\mathbb{Z}_{r_{k}}$ be the
invariant factor decomposition of $G$ and $e_{i}:=(0,\ldots,0,1,0,\ldots,0)$
for $i\in\{1,\ldots,k\}$, where $1$ occurs in the $i$th component.
For $g\in G$ we let $\hat{g}:=\left[\begin{matrix}\hat{g}_{1} & \cdots & \hat{g}_{k}\end{matrix}\right]\in\mathbb{Z}^{1\times k}$
such that $\hat{g_{i}}$ is the smallest nonnegative integer satisfying
\[
g=\sum\hat{g_{i}}e_{i}.
\]
\end{defn}
\begin{lem}
\label{lem:spr-1}If $g$ is an element of the abelian group $G$,
then $\spr(G/\langle g\rangle)\ge\spr(G)-1$. 
\end{lem}
\begin{proof}
Let $\mathbb{Z}_{r_{1}}\times\cdots\times\mathbb{Z}_{r_{k}}$ be the
invariant factor decomposition of $G$. Note that $1\ne r_{1}\mid r_{2}\mid\cdots\mid r_{k}$.
Then $G\cong\ag(A)$ with $A=\diag(r_{1},\ldots,r_{k})$ and $G/\langle g\rangle\cong\ag(B)\cong\ag(\snf(B))$
where
\[
B=\left[\begin{array}{c}
A\\
\hline \hat{g}
\end{array}\right]=\left[\begin{array}{ccc}
r_{1} &  & 0\\
 & \ddots\\
0 &  & r_{k}\\
\hat{g}_{1} & \cdots & \hat{g}_{k}
\end{array}\right].
\]
Let $\alpha_{1},\ldots,\alpha_{k}$ be the diagonal elements of $\snf(B)$.
If $\spr(G/\langle g\rangle)\le\spr(G)-2=k-2$, then $d_{1}(B)=\alpha_{1}=\alpha_{2}=1$.
This is impossible since it is easy to see that every $2\times2$
minor of $B$ is divisible by $r_{1}$ and so $\alpha_{2}=d_{2}(B)/d_{1}(B)=d_{2}(B)$
is divisible by $r_{1}$.
\end{proof}
\begin{prop}
\label{prop:spr}If $H$ is a subgroup and $g$ is an element of the
finite abelian group $G$, then
\[
\spr(G/\langle H\cup\{g\}\rangle)\ge\spr(G/H)-1.
\]
 
\end{prop}
\begin{proof}
Let $\mathbb{Z}_{r_{1}}\times\cdots\times\mathbb{Z}_{r_{k}}$ be the
invariant factor decomposition of $G$ and $H=\{h_{1},\ldots,h_{m}\}$.
Then $G\cong\ag(A)$ with $A=\diag(r_{1},\ldots,r_{k})$ and $G/H\cong\ag(B)\cong\ag(\tsnf(B))$
with
\[
B=\left[\begin{array}{c}
A\\
\hline \vphantom{\int^{\int}}\hat{H}
\end{array}\right]\text{ and }\hat{H}=\left[\begin{array}{c}
\widehat{\ h_{1}}\\
\hline \vdots\\
\hline \vphantom{\int^{\int^{\int}}}\widehat{h_{m}}
\end{array}\right].
\]
Note that $\widehat{h_{i}}=\left[\begin{matrix}\widehat{h_{i}}_{_{1}} & \cdots & \widehat{h_{i}}_{_{m}}\end{matrix}\right]$.
Applying elementary integer row and column operations gives
\[
G/\langle H\cup\{g\}\rangle\cong\ag\left[\begin{array}{c}
B\\
\hline \hat{g}
\end{array}\right]\cong\ag\left[\begin{array}{c|c}
I & 0\\
\hline 0 & \tsnf(B)\\
\hline 0 & \tilde{g}
\end{array}\right]\cong\ag\left[\begin{array}{c}
\tsnf(B)\\
\hline \tilde{g}
\end{array}\right]
\]
for some $\tilde{g}$. \textcolor{black}{So, $G/\langle H\cup\{g\}\rangle$
is isomorphic to a quotient of $G/H$ by a cyclic subgroup.}\textcolor{blue}{{}
}The result now follows as in the proof of Lemma~\ref{lem:spr-1}. 
\end{proof}
\begin{figure}
\begin{tabular}{ccccccc}
\includegraphics{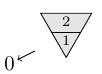} &  & \includegraphics{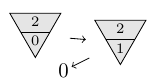} &  & \includegraphics{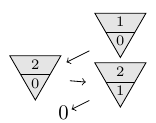} &  & \includegraphics{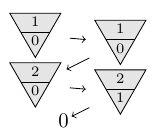}\tabularnewline
(a) $\spr(G)=1$ &  & (b) $\spr(G)=2$ &  & (c) $\spr(G)=3$ &  & (d) $\spr(G)\ge4$\tabularnewline
\end{tabular}

\caption{\label{fig:OddAbelian}Simplified structure diagrams for $\protect\gen(G)$
with $G=\mathbb{Z}_{p_{1}^{r_{1}}}\times\cdots\times\mathbb{Z}_{p_{k}^{r_{k}}}$
and odd primes $p_{1},\ldots,p_{k}$.}
\end{figure}

\begin{prop}
\label{prop:abelianGEN-oddPrimes}If $G=\mathbb{Z}_{p_{1}^{r_{1}}}\times\cdots\times\mathbb{Z}_{p_{k}^{r_{k}}}$
is an abelian group with odd primes $p_{1},\ldots,p_{k}$, then the
simplified structure diagram of $\gen(G)$ is one of the diagrams
shown in Figure~\ref{fig:OddAbelian}, and hence
\[
\gen(G)=\begin{cases}
*0, & \spr(G)=0\\
*2, & 1\le\spr(G)\le2\\
*1, & 3\le\spr(G)
\end{cases}.
\]
\end{prop}
\begin{proof}
The Frattini subgroup is $\Phi(G)\cong\mathbb{Z}_{p_{1}^{r_{1}-1}}\times\cdots\times\mathbb{Z}_{p_{k}^{r_{k}-1}}$\textcolor{red}{{}
}\textcolor{black}{by Propositions~\ref{prop:FrattiniProduct} and
\ref{prop:FrattiniCyclic}}, and so $G/\Phi(G)\cong\mathbb{Z}_{p_{1}}\times\cdots\times\mathbb{Z}_{p_{k}}$.
\textcolor{black}{One can show that this implies that every proper
subgroup of $G$ containing $\Phi(G)$ is an intersection subgroup.
}We use structural induction on the structure classes to show that
the simplified structure diagram is only dependent on $\spr(G)$ as
shown in Figure~\ref{fig:OddAbelian}, and 
\[
\type(X_{I})=\begin{cases}
t_{0}:=(1,0,0), & \spr(G/I)=0\\
t_{1}:=(1,2,1), & \spr(G/I)=1\\
t_{2}:=(1,2,0), & \spr(G/I)=2\\
t_{3}:=(1,1,0), & \spr(G)\ge\spr(G/I)\ge3
\end{cases}.
\]
The statement is clearly true when $\spr(G/I)=0$, that is, $I=G$.
Assume $\spr(G/I)=i\ge1$. Then $X_{I}$ has an option $X_{K}$ such
that $\spr(G/K)=i-1$ by Lemma~\ref{lem:lnf} \textcolor{black}{and
the proof of Proposition~\ref{prop:spr}.} Now, suppose $g\in G\setminus I$.
By Proposition~\ref{prop:spr}, $\spr(G/\langle I\cup\{g\}\rangle)\geq\spr(G/I)-1$,
which implies that if $X_{J}$ is an option of $X_{I}$, then $\spr(G/J)\in\{\spr(G/I),\spr(G/I)-1\}$. 

First, assume $j:=\spr(G/I)\in\{1,2,3\}$. Then $\otype(X_{I})=\{t_{j-1}\}$
or $\otype(X_{I})=\{t_{j-1},t_{j}\}$ by induction. In either case,
$\type(X_{I})=t_{j}$, and so $\Otype(X_{I})=\{t_{j-1},t_{j}\}$.
Now, assume $\spr(G/I)\ge4$. Then $\otype(X_{I})=\{t_{3}\}$ by induction.
Hence $\type(X_{I})=t_{3}$, and so $\Otype(X_{I})=\{t_{3}\}$. 
\end{proof}
\begin{figure}
\begin{tabular}{ccccccccc}
\includegraphics{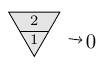} &  & \includegraphics{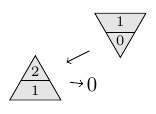} &  & \includegraphics{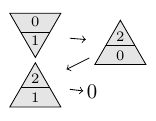} &  & \includegraphics{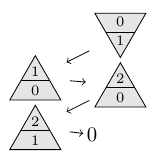} &  & \includegraphics{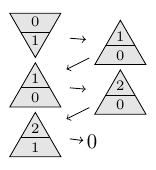}\tabularnewline
(a) $n=1$ &  & (b) $n=2$ &  & (c) $n=3$ &  & (d) $n=4$ &  & (e) $n\ge5$\tabularnewline
\end{tabular}

\caption{\label{fig:GenZ2ton}Simplified structure diagrams for $\protect\gen(\mathbb{Z}_{2}^{n})$.}
\end{figure}

\begin{example}
It is easy to check that the simplified structure diagram of $\gen(\mathbb{Z}_{2}^{n})$,
shown in Figure~\ref{fig:GenZ2ton}, follows a pattern somewhat similar
to that of $\gen(\mathbb{Z}_{p_{1}^{r_{1}}}\times\cdots\times\mathbb{Z}_{p_{k}^{r_{k}}})$,
shown in Figure~\ref{fig:OddAbelian}. Note that the only odd order
subgroup is the Frattini subgroup, which is the trivial subgroup.
\end{example}
The following result of \cite{Barnes} is an easy consequence of the
Third Isomorphism Theorem.
\begin{lem}
\label{lem:semiTerminalCyclic}If $I$ is a proper intersection subgroup
of the abelian group $G$, then $X_{I}$ is semi-terminal in $\gen(G)$
if and only if $G/I$ is cyclic.
\end{lem}
\begin{figure}
\begin{tabular}{ccc}
\includegraphics{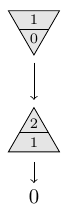} &  & \includegraphics{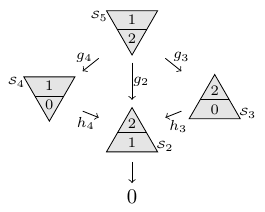}\tabularnewline
(a) $\gcd(m,k)=1$ &  & (b) $\gcd(m,k)\ne1$\tabularnewline
\end{tabular}\caption{\label{fig:mkodd}Simplified structure diagram for $\protect\gen(\mathbb{Z}_{2}\times\mathbb{Z}_{2}\times\mathbb{Z}_{m}\times\mathbb{Z}_{k})$
with odd $m$ and $k$.}
\end{figure}

\begin{prop}
\label{prop:abelianGEN-mkRelPrime}If $G=\mathbb{Z}_{2}\times\mathbb{Z}_{2}\times\mathbb{Z}_{m}\times\mathbb{Z}_{k}$
such that $m$ and $k$ are odd and relatively prime, then the simplified
structure diagram of $\gen(G)$ is the one shown in Figure~\ref{fig:mkodd}(a),
and hence $\gen(G)=*1$.
\end{prop}
\begin{proof}
Define the following collection of sets that form a partition of $\mathcal{Y}$:
\[
\begin{aligned}\mathcal{S}_{1} & :=\{X_{I}\in\mathcal{Y}\mid X_{I}\text{ is terminal}\};\\
\mathcal{S}_{2} & :=\{X_{I}\in\mathcal{Y}\mid X_{I}\text{ is semi-terminal}\};\\
\mathcal{S}_{3} & :=\{X_{I}\in\mathcal{Y}\mid X_{I}\text{ is non-terminal}\}.
\end{aligned}
\]
We use structural induction on the structure classes to show that
these sets are nonempty and are the type equivalence classes of the
structure classes, and that
\[
\type(X_{I})=\begin{cases}
(0,0,0), & X_{I}\in\mathcal{S}_{1}\\
(0,1,2), & X_{I}\in\mathcal{S}_{2}\\
(1,1,0), & X_{I}\in\mathcal{S}_{3}
\end{cases}.
\]

First, note that since the orders of $\mathbb{Z}_{2}\times\mathbb{Z}_{2}$
and $\mathbb{Z}_{m}\times\mathbb{Z}_{k}$ are relatively prime, every
subgroup is of the form $I=H\times K$ for some $H\leq\mathbb{Z}_{2}\times\mathbb{Z}_{2}$
and $K\leq\mathbb{Z}_{m}\times\mathbb{Z}_{k}$ by Proposition~\ref{prop:SubgroupRelPrime}. 

We show that if $I$ is an intersection subgroup, then $X_{I}$ is
even if and only if $X_{I}$ is semi-terminal. Suppose that $X_{I}$
is even. Then $H$ is nontrivial, which implies that $(\mathbb{Z}_{2}\times\mathbb{Z}_{2})/H$
is isomorphic to $\langle0\rangle$ or $\mathbb{Z}_{2}$. Since $m$
and $k$ are relatively prime, $\mathbb{Z}_{m}\times\mathbb{Z}_{k}$
is cyclic, and hence $(\mathbb{Z}_{m}\times\mathbb{Z}_{k})/K$ is
cyclic, as well. This implies that 
\[
G/I=G/(H\times K)\cong(\mathbb{Z}_{2}\times\mathbb{Z}_{2})/H\times(\mathbb{Z}_{m}\times\mathbb{Z}_{k})/K
\]
is cyclic since the orders of $(\mathbb{Z}_{2}\times\mathbb{Z}_{2})/H$
and $(\mathbb{Z}_{m}\times\mathbb{Z}_{k})/K$ are relatively prime
and each group is cyclic. Hence $X_{I}$ is semi-terminal by Lemma~\ref{lem:semiTerminalCyclic}.
On the other hand, if $X_{I}$ is semi-terminal, reversing the above
argument presents no difficulties, so $X_{I}$ is even.

It is clear that $\mathcal{S}_{1}=\{X_{G}\}\neq\emptyset$ and $\type(X_{G})=(0,0,0)$. 

The maximal subgroup $R=\langle0\rangle\times\mathbb{Z}_{2}\times\mathbb{Z}_{m}\times\mathbb{Z}_{k}$
is even. Hence $X_{R}\in\mathcal{S}_{2}$, and so $\mathcal{S}_{2}\neq\emptyset$.
Let $X_{I}\in\mathcal{S}_{2}$. Then by the above, $X_{I}$ is even.
Hence $\otype(X_{I})$ is either $\{(0,0,0)\}$ or $\{(0,0,0),(0,1,2)\}$
by induction. In either case, $\type(X_{I})=(0,1,2)$, and so $\Otype(X_{I})=\{(0,0,0),(0,1,2)\}$.

For the final case, observe that $X_{\Phi(G)}\in\mathcal{S}_{3}$
since $\Phi(G)$ is odd by Propositions~\ref{prop:FrattiniProduct}
and \ref{prop:FrattiniCyclic}, and so $\mathcal{S}_{3}\neq\emptyset$.
Let $X_{I}\in\mathcal{S}_{3}$. Then $X_{I}$ must be odd. Moreover,
$I\cup\{(1,0,0,0)\}$ is an option of $I$ that lies in a semi-terminal
structure class. Hence $\otype(X_{I})$ is either $\{(0,1,2)\}$ or
$\{(0,1,2),(1,1,0)\}$ by induction. In either case, $\type(X_{I})=(1,1,0)$,
and hence $\Otype(X_{I})=\{(0,1,2),(1,1,0)\}$.

It follows that the simplified structure diagram of $\gen(G)$ is
the one shown in Figure~\ref{fig:mkodd}(a), and so $\gen(G)=*1$.
\end{proof}
\begin{prop}
\label{prop:abelianGEN-mkNotRelPrime}If $G=\mathbb{Z}_{2}\times\mathbb{Z}_{2}\times\mathbb{Z}_{m}\times\mathbb{Z}_{k}$
such that $m$ and $k$ are odd and not relatively prime, then the
simplified structure diagram of $\gen(G)$ is the one shown in Figure~\ref{fig:mkodd}(b),
and hence $\gen(G)=*1$.
\end{prop}
\begin{proof}
Define the following collection of sets that form a partition of $\mathcal{Y}$:
\[
\begin{aligned}\mathcal{S}_{1} & :=\{X_{I}\in\mathcal{Y}\mid X_{I}\text{ is terminal}\};\\
\mathcal{S}_{2} & :=\{X_{I}\in\mathcal{Y}\mid X_{I}\text{ is semi-terminal}\};\\
\mathcal{S}_{3} & :=\{X_{I}\in\mathcal{Y}\mid X_{I}\text{ is even and non-terminal}\};\\
\mathcal{S}_{4} & :=\{X_{I}\in\mathcal{Y}\mid X_{I}\text{ is odd and non-terminal without any even non-terminal option}\};\\
\mathcal{S}_{5} & :=\{X_{I}\in\mathcal{Y}\mid X_{I}\text{ is odd and non-terminal with an even non-terminal option}\}.
\end{aligned}
\]
We use structural induction on the structure classes to show that
these sets are nonempty and are the type equivalence classes of the
structure classes and that
\[
\type(X_{I})=\begin{cases}
(0,0,0), & X_{I}\in\mathcal{S}_{1}\\
(0,1,2), & X_{I}\in\mathcal{S}_{2}\\
(0,0,2), & X_{I}\in\mathcal{S}_{3}\\
(1,1,0), & X_{I}\in\mathcal{S}_{4}\\
(1,1,2), & X_{I}\in\mathcal{S}_{5}
\end{cases}.
\]

It is clear that $\mathcal{S}_{1}=\{X_{G}\}\ne\emptyset$ and $\type(X_{G})=(0,0,0)$.

Next, consider the even maximal subgroup $R=\langle0\rangle\times\mathbb{Z}_{2}\times\mathbb{Z}_{m}\times\mathbb{Z}_{k}$.
Then $X_{R}\in\mathcal{S}_{2}$ since $\langle R\cup\{(1,0,0,0)\}\rangle=G$,
and so $\mathcal{S}_{2}\neq\emptyset$. Let $X_{I}\in\mathcal{S}_{2}$.
Note that the orders of $\mathbb{Z}_{2}\times\mathbb{Z}_{2}$ and
$\mathbb{Z}_{m}\times\mathbb{Z}_{k}$ are relatively prime, so $I=H\times K$
for some $H\leq\mathbb{Z}_{2}\times\mathbb{Z}_{2}$ and $K\leq\mathbb{Z}_{m}\times\mathbb{Z}_{k}$
by Proposition~\ref{prop:SubgroupRelPrime}. Since $X_{I}$ is semi-terminal,
the quotient
\[
G/I=G/(H\times K)\cong(\mathbb{Z}_{2}\times\mathbb{Z}_{2})/H\times(\mathbb{Z}_{m}\times\mathbb{Z}_{k})/K
\]
is cyclic by Lemma~\ref{lem:semiTerminalCyclic}. This happens only
if
\[
H\in\{\langle0\rangle\times\mathbb{Z}_{2},\mathbb{Z}_{2}\times\langle0\rangle,\langle(1,1)\rangle,\mathbb{Z}_{2}\times\mathbb{Z}_{2}\},
\]
which shows that $X_{I}$ must be even. Hence $\otype(X_{I})$ is
either $\left\{ (0,0,0)\right\} $ or $\left\{ (0,0,0),(0,1,2)\right\} $
by induction since $X_{I}$ is semi-terminal. In either case, $\type(X_{I})=(0,1,2)$,
and hence $\Otype(X_{I})=\left\{ (0,0,0),(0,1,2)\right\} $.

For the third case, consider the even subgroups $R_{1}=\mathbb{Z}_{2}\times\langle0\rangle\times\langle0\rangle\times\langle0\rangle$,
$R_{2}=\langle0\rangle\times\mathbb{Z}_{2}\times\langle0\rangle\times\langle0\rangle$,
and $R_{3}=\langle(1,1)\rangle\times\langle0\rangle\times\langle0\rangle$.
Let $X_{I_{i}}$ be the structure class containing $R_{i}$. Since
$m$ and $k$ are not relatively prime, we can choose $m'$ and $k'$
such that $m/m'=k/k'$ is a prime. Then $I_{i}$ is contained in the
sets
\[
\begin{aligned}M & :=\mathbb{Z}_{2}\times\mathbb{Z}_{2}\times\langle m/m'\rangle\times\mathbb{Z}_{k}\cong\mathbb{Z}_{2}\times\mathbb{Z}_{2}\times\mathbb{Z}_{m'}\times\mathbb{Z}_{k}\\
N & :=\mathbb{Z}_{2}\times\mathbb{Z}_{2}\times\mathbb{Z}_{m}\times\langle k/k'\rangle\cong\mathbb{Z}_{2}\times\mathbb{Z}_{2}\times\mathbb{Z}_{m}\times\mathbb{Z}_{k'}.
\end{aligned}
\]
By Proposition~\ref{prop:maximalIndex}, $M$ and $N$ are maximal
subgroups. Since $M\cap N\cong\mathbb{Z}_{2}\times\mathbb{Z}_{2}\times\mathbb{Z}_{m'}\times\mathbb{Z}_{k'}$,
$G/M\cap N\cong\mathbb{Z}_{m/m'}^{2}$ is not cyclic. Hence each $G/I_{i}$
is not cyclic. So each $X_{I_{i}}$ is even and non-terminal by Lemma~\ref{lem:semiTerminalCyclic}.
Thus each $R_{i}$ is an element of an even non-terminal structure
class, which shows that $\mathcal{S}_{3}\neq\emptyset$. Let $X_{I}\in\mathcal{S}_{3}$.
We show that $X_{I}$ has an option in $\mathcal{S}_{2}$. Since $I$
is even, at least one of $R_{1}$, $R_{2}$, or $R_{3}$ is a subgroup
of $I$. Let $h_{3}=(0,0,1,0)$. Since $\gcd(2,k)=1$, each $G/\langle R_{i}\cup\{h_{3}\}\rangle\cong\mathbb{Z}_{2}\times\mathbb{Z}_{m}$
is cyclic. This implies that $G/\langle I\cup\{h_{3}\}\rangle$ is
cyclic, and hence $I\cup\{h_{3}\}$ is an element of an even semi-terminal
structure class. It is clear that every option of $X_{I}$ is in $\mathcal{S}_{2}\cup\mathcal{S}_{3}$.
Therefore, $\otype(X_{I})$ is either $\left\{ (0,1,2)\right\} $
or $\left\{ (0,1,2),(0,0,2)\right\} $ by induction. In either case,
$\type(X_{I})=(0,0,2)$, and so $\Otype(X_{I})=\{(0,1,2),(0,0,2)\}$.

For the fourth case, consider the odd subgroup $Q:=\langle0\rangle\times\langle0\rangle\times\mathbb{Z}_{m}\times\langle0\rangle$.
Then $Q$ is contained in an odd structure class $X_{J}$ since $Q$
is a subset of the maximal subgroups $\langle0\rangle\times\mathbb{Z}_{2}\times\mathbb{Z}_{m}\times\mathbb{Z}_{k}$
and $\mathbb{Z}_{2}\times\langle0\rangle\times\mathbb{Z}_{m}\times\mathbb{Z}_{k}$.
We show that $X_{J}\in\mathcal{S}_{4}$. For a contradiction, assume
that $Q\cup\{h\}$ is in a structure class inside $\mathcal{S}_{3}$
for some $h$. Then $\langle Q\cup\{h\}\rangle\cong\mathbb{Z}_{2}\times\mathbb{Z}_{m}\times U$,
where $U\leq\mathbb{Z}_{k}$. This implies that $G/\langle Q\cup\{h\}\rangle\cong\mathbb{Z}_{2}\times(\mathbb{Z}_{k}/U)$,
which is cyclic since $k$ is odd. Then $Q\cup\{h\}$ is in a structure
class in $\mathcal{S}_{2}$ by Lemma~\ref{lem:semiTerminalCyclic}.
This is a contradiction, and so $\mathcal{S}_{4}\neq\emptyset$. Now
consider $X_{I}\in\mathcal{S}_{4}$. If $h_{4}$ is any element of
$G$ with even order, then $I\cup\{h_{4}\}$ is contained in an even
semi-terminal structure class by the definition of $\mathcal{S}_{4}$.
So $X_{I}$ has an option in $\mathcal{S}_{2}$. Next, we show that
$X_{I}$ has no option in $\mathcal{S}_{5}$. For a contradiction,
suppose $X_{I}$ has an option $X_{J}$ in $\mathcal{S}_{5}$. By
the definition of $\mathcal{S}_{5}$, $X_{J}$ has an option $X_{H}$
in $\mathcal{S}_{3}$. Let $t$ be an element of order 2 in $H$ and
let $I\cup\{t\}\in X_{K}$. Then $K$ is even and $K\leq H\leq G$.
So $X_{K}$ is an option of $X_{I}$ in $\mathcal{S}_{3}$, which
is a contradiction. Hence any option of $X_{I}$ is in $\mathcal{S}_{2}$
or $\mathcal{S}_{4}$. Thus, $\otype(X_{I})$ is either $\left\{ (0,1,2)\right\} $
or $\left\{ (0,1,2),(1,1,0)\right\} $ by induction. In either case,
$\type(X_{I})=(1,1,0)$ and so $\Otype(X_{I})=\{(0,1,2),(1,1,0)\}$.

To show that $\mathcal{S}_{5}\neq\emptyset$, consider the empty position
$\emptyset$, which is an element of the odd structure class $X_{\Phi(G)}$.
Since $G$ is not cyclic, $X_{\Phi(G)}$ must be a non-terminal structure
class. Let $g_{3}:=(1,0,0,0)$. The empty position has the option
$\{g_{3}\}$ that belongs to a structure class in $\mathcal{S}_{3}$
by Lemma~\ref{lem:semiTerminalCyclic}, since
\[
G/\langle\{g_{3}\}\rangle=G/(\mathbb{Z}_{2}\times\langle0\rangle\times\langle0\rangle\times\langle0\rangle)\cong\mathbb{Z}_{2}\times\mathbb{Z}_{m}\times\mathbb{Z}_{k}
\]
is not cyclic as $\gcd(m,k)\ne1$. Hence $X_{\Phi(G)}\in\mathcal{S}_{5}$.
Next, let $X_{I}\in\mathcal{S}_{5}$. Let $g_{2}=(1,0,1,0)$ and $g_{4}=(0,0,1,0)$.
Then $I\cup\{g_{2}\}$ belongs to an even semi-terminal structure
class by Lemma~\ref{lem:semiTerminalCyclic} since $G/\langle I\cup\{g_{2}\}\rangle$
is cyclic. This implies that $X_{I}$ has an option in $\mathcal{S}_{2}$.
Lastly, we will show that $I\cup\{g_{4}\}$ is an option of $I$ in
a structure class in $\mathcal{S}_{4}$. We see that $Q$ is a subgroup
of $\langle I\cup\{g_{4}\}\rangle$. We know from the fourth case
that $\langle g_{4}\rangle=Q\in X_{J}\in\mathcal{S}_{4}$. Hence the
odd structure class containing $I\cup\{g_{4}\}$ belongs to $\mathcal{S}_{4}$.
As a consequence, $\otype(X_{I})$ is either $\{(0,0,2),(1,1,0),(0,1,2)\}$
or $\{(0,0,2),(1,1,0),(0,1,2),(1,1,2)\}$ by induction. In either
case, $\type(X_{I})=(1,1,2)$, and hence $\Otype(X_{I})=\{(0,0,2),(1,1,0),(0,1,2),(1,1,2)\}$.

It follows that the simplified structure diagram of $\gen(G)$ is
the one shown in Figure~\ref{fig:mkodd}(b), which implies that $\gen(G)=*1$.
\end{proof}
It is interesting to notice that the simplified structure diagram
from the previous theorem is the same as that of $\gen(\mathbb{D}_{4k+2})$
shown in Figure~\ref{fig:structureDiagsDihedral}(f).
\begin{cor}
If $G$ is a finite abelian group, then
\[
\gen(G)=\begin{cases}
*2, & |G|\text{ is odd and }1\le\spr(G)\le2\\
*1, & |G|\text{ is odd and }\spr(G)\ge3\\
*2, & G\cong\mathbb{Z}_{2}\\
*1, & G\cong\mathbb{Z}_{4k}\text{ with }k\ge1\\
*4, & G\cong\mathbb{Z}_{4k+2}\text{ with }k\ge1\\
*1, & G\cong\mathbb{Z}_{2}\times\mathbb{Z}_{2}\times\mathbb{Z}_{m}\times\mathbb{Z}_{k}\text{ for }m,k\text{ odd}\\
*0, & \text{else}
\end{cases}.
\]
\end{cor}
\begin{proof}
By \cite[Section 3.2]{Barnes}, the first player wins $\gen(G)$ exactly
when $G$ is cyclic, odd, or isomorphic to $\mathbb{Z}_{2}\times\mathbb{Z}_{2}\times\mathbb{Z}_{m}\times\mathbb{Z}_{k}$
with odd $m$ and $k$. We already proved the result in these cases.
In every other case the second player wins, so in these cases $\gen(G)=*0$.
\end{proof}

\section{Symmetric and alternating groups\label{sec:SymmetricAlternating}}

\begin{table}
\begin{centering}
\begin{tabular}{|c|c|c|c|c|c|c|c|c|c|c|c|}
\cline{1-4} \cline{6-12} 
$n$ & $2$ & $3$ & $\ge4$ & $\qquad$ & $n$ & $3$ & $4$ & $5$ & $6$ & $7$ & $8$\tabularnewline
\cline{1-4} \cline{6-12} 
$\dng(S_{n})$ & $*1$ & $*3$ & $*0$ &  & $\dng(A_{n})$ & $*3$ & $*3$ & $*0$ & $*0$ & $*0$ & $*0$\tabularnewline
\cline{1-4} \cline{6-12} 
\end{tabular}\bigskip
\par\end{centering}
\caption{\label{table:DNGSnJoint}Nimbers for $\protect\dng(S_{n})$ and $\protect\dng(A_{n})$.}
\end{table}

According to \cite[Section 2.3]{Barnes}, the second player wins $\dng(S_{n})$
for $n\geq4$. Hence $\dng(S_{n})=*0$ for $n\ge4$. We already studied
the remaining cases. In particular, we know that $\dng(S_{2})=\dng(\mathbb{Z}_{2})=*1$
and $\dng(S_{3})=\dng(\mathbb{D}_{3})=*3$. These results are summarized
in Table~\ref{table:DNGSnJoint}.

The $\dng(A_{n})$ games remain a bit of a mystery. There is an incomplete
analysis of $\dng(A_{n})$ in \cite[Section 2.4]{Barnes} that involves
some fairly deep group-theoretic results. By Proposition~\ref{prop:DNGodd},
we know $\dng(A_{3})=\dng(\mathbb{Z}_{3})=*1$. Some additional computer
calculated values are listed in Table~\ref{table:DNGSnJoint}.

By \cite[Section 3.4]{Barnes}, the first player wins $\gen(S_{n})$
for $n\ge5$ and $\gen(A_{n})$ for all $n\ge4$. On the other hand,
$\gen(S_{2})=\gen(\mathbb{Z}_{2})=*2$, $\gen(S_{3})=\gen(\mathbb{D}_{3})=*3$,
$\gen(A_{3})=\gen(\mathbb{Z}_{3})=*2$, and computer calculations
show that $\gen(S_{4})=*0$. 

\begin{figure}
\begin{tabular}{ccccc}
\includegraphics{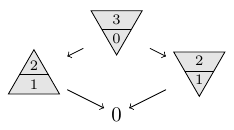} &  & \includegraphics{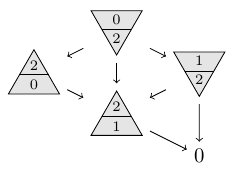} &  & \includegraphics{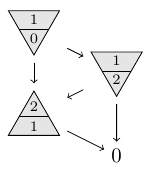}\tabularnewline
(a) $\gen(A_{4})$ &  & (b) $\gen(S_{4})$ &  & (c) $\gen(S_{n})$, $\gen(A_{n})$, $n\ge5$\tabularnewline
\end{tabular}

\vspace{2mm}
\begin{centering}
\begin{tabular}{|c|c|c|c|c|}
\hline 
$n$ & $2$ & $3$ & $4$ & $\ge5$\tabularnewline
\hline 
$\gen(S_{n})$ & $*2$ & $*3$ & $*0$ & $*1$?\tabularnewline
\hline 
$\gen(A_{n})$ &  & $*2$ & $*3$ & $*1$?\tabularnewline
\hline 
\end{tabular}
\par\end{centering}
\caption{\label{fig:Sn}Known and conjectured values and simplified structure
diagrams for $\protect\gen(S_{n})$ and $\protect\gen(A_{n})$.}
\end{figure}

\begin{conjecture}
\label{conj:SnAn}For $n\ge5$, we have $\gen(S_{n})=*1=\gen(A_{n})$.
\end{conjecture}
We have verified the conjecture for $n\in\{5,6,7,8\}$ using our software.
The conjectured simplified structure diagrams are shown in Figure~\ref{fig:Sn}. 

\section{Further questions\label{sec:Questions}}

We conclude with a few open problems that may not be out of reach.
\begin{enumerate}
\item Is it possible to have a directed cycle in the simplified structure
digraph? This never happens for the groups we have considered, but
are there groups for which this is possible?
\item Does Conjecture~\ref{conj:SnAn} about $\gen(S_{n})$ and $\gen(A_{n})$
hold? Structural induction on the structure classes could work to
show that the simplified structure diagram of Figure~\ref{fig:Sn}(c)
is correct. 
\item What are the remaining nim-numbers for $\dng(A_{n})$? The partial
results of \cite[Theorem 2]{Barnes} are quite complicated. Settling
this question likely requires some sophisticated group theory arguments.
\item What are the nim-numbers of other families of groups? In particular,
what are the nim-numbers of generalized dihedral groups of the form
$A\rtimes\mathbb{Z}_{2}$ where $A$ is a finite abelian group?
\item Are there any general results regarding quotient groups and direct
products of groups? A positive answer would be very helpful in the
study of more complicated groups.
\item Does Conjecture~\ref{conj:nimEvenGEN01234} about the spectrum of
$\gen(G)$ hold? Is the set $\{\nim(\gen(G))\mid G\text{ is a finite group}\}$
at least bounded? 
\item Can the structure diagram approach be generalized to study $\dng$
and $\gen$ on other algebraic structures with generators (e.g., semigroups,
inverse semigroups) and compute their corresponding nim-numbers?
\item The Sprague-Grundy theory is generalized to infinite loopy games in
\cite{FraenkelPerlLoopy,SmithLoopy}. A recent description of the
theory can be found in \cite[IV.4]{SiegelBook}. Can our techniques
be generalized to find the loopy nim-numbers of some families of infinite
groups? 
\end{enumerate}
\bibliographystyle{amsplain}
\bibliography{roughdraftbib,game}

\providecommand{\bysame}{\leavevmode\hbox to3em{\hrulefill}\thinspace}
\providecommand{\MR}{\relax\ifhmode\unskip\space\fi MR }
\providecommand{\MRhref}[2]{%
  \href{http://www.ams.org/mathscinet-getitem?mr=#1}{#2}
}
\providecommand{\href}[2]{#2}
\begin{thebibliography}{10}

\bibitem{albert2007lessons}
Michael Albert, Richard Nowakowski, and David Wolfe, \emph{Lessons in play: an
  introduction to combinatorial game theory}, AMC \textbf{10} (2007), 12.

\bibitem{anderson.harary:achievement}
M.~Anderson and F.~Harary, \emph{Achievement and avoidance games for generating
  abelian groups}, Internat. J. Game Theory \textbf{16} (1987), no.~4,
  321--325.

\bibitem{Barnes}
F.~W. Barnes, \emph{Some games of {F}. {H}arary, based on finite groups}, Ars
  Combin. \textbf{25} (1988), no.~A, 21--30, Eleventh British Combinatorial
  Conference (London, 1987).

\bibitem{brandenburg:algebraicGames}
Martin Brandenburg, \emph{Algebraic games playing with groups and rings},
  Inter. J. of Game Theory (2017), 1--34.

\bibitem{Dixon}
John~D. Dixon, \emph{Problems in group theory}, Blaisdell Publishing Co. Ginn
  and Co., Waltham, Mass.-Toronto, Ont.-London, 1967.

\bibitem{Dlab1960}
Vlastimil Dlab, \emph{{The Frattini subgroups of abelian groups}}, Czechoslovak
  Mathematical Journal \textbf{10} (1960), no.~1.

\bibitem{DummitFoote}
David~S. Dummit and Richard~M. Foote, \emph{Abstract algebra}, third ed., John
  Wiley \& Sons, Inc., Hoboken, NJ, 2004.

\bibitem{WEB}
Dana~C. Ernst and N{\'a}ndor Sieben, \emph{Companion web site}, 2013,
  \texttt{http://jan.ucc.nau.edu/ns46/GroupGenGame}.

\bibitem{FraenkelPerlLoopy}
A.~S. Fraenkel and Y.~Perl, \emph{Constructions in combinatorial games with
  cycles}, Infinite and finite sets ({C}olloq., {K}eszthely, 1973; dedicated to
  {P}. {E}rd{\H o}s on his 60th birthday), {V}ol. {II}, North-Holland,
  Amsterdam, 1975, pp.~667--699. Colloq. Math. Soc. Jan\'os Bolyai, Vol. 10.

\bibitem{GAP}
The GAP~Group, \emph{{GAP -- Groups, Algorithms, and Programming, Version
  4.6.4}}, 2013.

\bibitem{Jacobson}
Nathan Jacobson, \emph{Basic algebra. {I}}, second ed., W. H. Freeman and
  Company, New York, 1985.

\bibitem{Rose}
J.S. Rose, \emph{A course on group theory}, Dover Publications, Inc., New York,
  1994.

\bibitem{SiegelBook}
Aaron~N. Siegel, \emph{Combinatorial game theory}, Graduate Studies in
  Mathematics, vol. 146, American Mathematical Society, Providence, RI, 2013.

\bibitem{SmithLoopy}
Cedric A.~B. Smith, \emph{Graphs and composite games}, J. Combinatorial Theory
  \textbf{1} (1966), 51--81.

\bibitem{Suzuki}
Michio Suzuki, \emph{Group theory. {I}}, Grundlehren der Mathematischen
  Wissenschaften, vol. 247, Springer-Verlag, Berlin, 1982.

\end{thebibliography}

\end{document}